\documentclass[12pt,twoside]{article}
\usepackage{amsmath,amssymb}
\usepackage{graphicx}
\usepackage{color}
\usepackage{subcaption}
\usepackage{float}
\usepackage[bookmarksnumbered=true,colorlinks,linkcolor=black]{hyperref}
\hypersetup{hypertex=true,colorlinks=true,linkcolor=red,anchorcolor=blue,citecolor=red}
\hypersetup{CJKbookmarks=true}

\allowdisplaybreaks[4]

\newtheorem{lm}{Lemma}[section]
\newtheorem{thm}{Theorem}[section]
\newtheorem{rmk}{Remark}[section]

\newcounter{saveeqn}%

\makeatletter
\evensidemargin
\oddsidemargin
\makeatother

\title{\Large\bf
Global center of polynomial Newton system and its non-isochronicity
\thanks{Jun Zhang is supported by NSFC \#12101087, CSC \#202108515020 and NSFSPC \#2024NSFSC1400,
Weinian Zhang is supported by National Key R\&D Program of China (2022YFA1005900) and NSFC \#12171336.
}
}

\author{{\sc Colin Christopher}\,$^a$,~~{\sc Jun Zhang}\,$^b$,~~{\sc Weinian Zhang}\,$^c$
\\
\\
$^a${\small
School of Engineering, Computing and Mathematics, University of Plymouth}\\
{\small Plymouth, Devon PL4 8AA, UK}
\\
$^b${\small
School of Mathematical Sciences \& Sichuan Geomath Key Lab }\\
{\small Chengdu University of Technology, Sichuan 610059, P. R. China}
\\
$^c${\small
School of Mathematics, Sichuan University}\\
{\small Chengdu, Sichuan 610064, P. R. China}
}
\date{}

\begin{document}
\maketitle

\begin{abstract}
Using a new compactification (toroidal compactification) and desingularization,
we obtain a complete characterization of
monodromy at infinity for polynomial Newton system of arbitrary degree,
in which we establish an equivalence between the monodromy
and
the non-existence of $\frac{1}{2}$-fractional formal invariant curves.
Combining the complete characterization with
either Darboux integrability
or algebraic reducibility of local centers,
we obtain conditions for all cases of global center.
Furthermore,
investigating the asymptotic behavior of the period function of orbits near infinity,
we prove the non-isochronicity for the global center,
which consequently solves an open problem proposed by Conti.

\vskip 0.2cm


\vskip 0.2cm
{\bf AMS (2020) subject classification:}
34C05, 
34C25, 
34D23. 

\end{abstract}

\baselineskip 15pt    
\parskip 10pt         

\thispagestyle{empty}
\setcounter{page}{1}

\section{Introduction}
\setcounter{equation}{0}
\setcounter{lm}{0}
\setcounter{thm}{0}
\setcounter{rmk}{0}
\setcounter{df}{0}
\setcounter{cor}{0}
\setcounter{pro}{0}

By Newton's second law,
the motion of an object can be described by the second order differential equation $\ddot x=F(x,\dot x,t)$
(or equivalently the planar system $\dot x=y$ and $\dot y=F(x,y,t)$),
called the {\it Newton equation} (or the {\it Newton system}) in \cite[p.8]{Arnold89},
where $x$ is the position of the object and
$y$ is the velocity.
If we simply consider the autonomous case (i.e., $F$ is independent of $t$) and
$F$ is of the polynomial form,
the system can be presented as
\begin{align}
\dot x=y,~~~
\dot y=P_0(x)+P_1(x)y+\cdots+P_m(x)y^m,
\label{equ:Newton}
\end{align}
where integer $m\ge 0$, each $P_i(x)$ is a polynomial in $x$ and $P_m(x) \not\equiv 0$.
This system becomes
the well-known potential system (\cite{CJ89}) for $m=0$,
Li\'enard system (\cite{Cherkas76}) for $m=1$ and
Cherkas system (\cite{Cherkas76}) for $m=2$.

Oscillations in planar differential systems are closely related to center,
a concept introduced by Poincar\'e,
which refers to an equilibrium having a vicinity filled with nontrivial periodic orbits.
Many works (e.g. \cite{Cherkas76,Chris} for Li\'enard system and
\cite{Cherkas76,CS08} for Cherkas system)
are contributed to determination of centers.
One of interesting topics on centers is {\it global center},
i.e., the phase plane is full of nontrivial periodic orbits except for the center equilibrium,
because global center is not only an important research object 
in the study of the Arnol'd-Hilbert's 16th problem (\cite{DL03, FHX})
but also related to the real planar Jacobian conjecture (\cite{Gav97,Saba98}).
Since Galeotti and Villarini (\cite{GV92}) and also Liang (\cite{Liang92}) proved in 1992 that
a polynomial differential system is of odd degree
if the system has a global center,
great efforts have been made to identify all the polynomial differential systems of odd degree having a global center.
For polynomial rigid systems (i.e., the angular velocity is a constant),
Conti (\cite{Conti94}) proved that such a system has a global center if and only if
it is a linear system.
For homogeneous polynomial systems,
Conti (\cite[section~4]{Conti98}) gave a necessary and sufficient condition.
Later, this result was generalized in \cite{LLYZ} to quasi-homogeneous case.
For cubic Hamiltonian system without quadratic term,
authors in \cite{CLV14j,CLV14a} characterized
the non-degenerate global center and the nilpotent degenerate one.
For polynomial Rayleigh-Duffing system,
it is proved in \cite{CZZ} that there are no global centers.
One can refer to \cite{CFZ} and \cite{CL24} for details
on global center of the Kukles system.
For polynomial Newton systems,
a necessary and sufficient condition for the global center of system~\eqref{equ:Newton}
with $m=0$ is obtained in \cite{GG15},
and $m=1$ in \cite{LV22} for the non-degenerate case
and also in \cite{CLZ} for both the non-degenerate case and the nilpotent degenerate case.
However, up to now,
condition about global center was not given yet for
system~\eqref{equ:Newton} with a larger $m\ge 2$.

Knowing a center,
one may take concern to its isochronicity, i.e.,
every periodic orbit in the periodic annulus has the same period.
The study of isochronicity goes back to Huygens
who investigated the isochronous oscillation of cycloidal pendulum.
At present the existence of isochronous centers for many classes of systems,
such as potential system (\cite{CJ89}),
Li\'enard system (\cite{CD04}),
Cherkas system (\cite{Saba04}) and
linear plus homogeneous system (\cite{LV11}),
has been investigated.
However, global isochronicity has received less attention.
Conti (\cite[p.229]{Conti98}) gave
an example of Hamiltonian system of general odd degree
having an isochronous global center.
So he further proposed the problem:
{\it Identify polynomial systems of odd degree having an isochronous global center}.
This problem was solved by the answer that
there are no isochronous global centers for
nonlinear potential system (\cite{CJ89}),
separable Hamiltonian system (\cite{CGM}),
rigid system (\cite{Conti94})
and polynomial system with linear plus homogeneous nonlinearities (\cite{LSZ}).
On the contrary,
authors in \cite{LSZ} obtained all cubic polynomial systems having an isochronous global center,
and authors in \cite{CMV99} characterized
isochronous global center for Hamiltonian system of the form
$H(x,y)=A(x)+B(x)y+C(x)y^2$ with polynomials $A$, $B$ and $C$.
Continuing the above,
one may further ask:
{\it
Can the global center of system~\eqref{equ:Newton} with $m\ge 2$
be isochronous?}

In this paper we give necessary and sufficient conditions for
global center of polynomial Newton system~\eqref{equ:Newton}
and further prove its non-isochronicity.
Since a global center is related to monodromy at infinity,
in section 2 we compactify the phase plane to investigate such a monodromy
for preliminaries.
Usually, compactification is made to a disk or a sphere (\cite[Chapter~5]{DLA}).
Here we consider a torus as the compactification
under which the system at infinity is of the same form as \eqref{equ:Newton},
which is simpler than those under the usual compactification
so that desingularization of degenerate equilibria at infinity gets easier.
Under toroidal compactification,
we can prove easily that if system~\eqref{equ:Newton} is monodromic at infinity
then $m\le 2$, i.e., it is a Cherkas system.
Further,
desingularizing degenerate equilibria at infinity by quasi-homogenerous blow-ups
associated with Newton polygons,
we obtain a complete characterization of the monodromy at infinity
for the Cherkas system.
Moreover,
we establish an equivalence between the monodromy at infinity
and the non-existence of $\frac{1}{2}$-fractional formal invariant curves at infinity,
which can be identified by a finitely determinable system of polynomial equations.
It is worth mentioning that
except for those $\frac{1}{2}$-fractional formal invariant curves,
there would be ones of more complicated forms
if the Cherkas system is not monodromic at infinity, see Remark~\ref{rmk:node-ln}.
In section 3,
combining the result of monodromy at infinity
with either the Darboux integrability or the algebraic reducibility of Cherkas centers,
we characterize all cases of global center for Cherkas system.
Section 4 is devoted to the non-isochronicity
by investigating the asymptotic behavior of the period function of orbits near infinity.
For this purpose,
we consider a closed rectangle as a compactifycation of the real plane,
and then the boundary is a polycycle consisting of a finite union of equilibria
and integral curves.
After desingularization,
the boundary becomes a polycycle with no more than $4(n+3)$ vertices,
where $n:=\max\{\deg P_0,\deg P_1,\deg P_2\}$.
Then we scrutinize orders of poles and zeros of the desingularized vector field
at each vertex and along each side between two adjacent vertices to
determine cases that the period of orbit near infinity approaches to zero or infinity.
In the case that the period approaches to a nonzero constant,
we find that the polynomial Cherkas system can be transformed into
an analytic Li\'enard system with odd damping and restoration.
Further, the non-isochronicity can be derived by
applying Liouville's theorem on integration in terms of elementary functions
to show the invalidity of the isochronous center condition
for the analytic Li\'enard system (\cite{CD04}).
Thus,
the isochronous global center problem proposed by Conti is solved
for Newton system \eqref{equ:Newton}.
As an application,
in section 5 we generalize the global center result in \cite{CL24}
on homogeneous Kukles system of degree 5 to arbitrary degree.
Moveover,
we provide a simple proof for global center of polynomial Li\'enard system
to show the convenience of toroidal compactification.

\section{Monodromy at infinity}
\setcounter{equation}{0}
\setcounter{lm}{0}
\setcounter{thm}{0}
\setcounter{rmk}{0}
\setcounter{df}{0}
\setcounter{cor}{0}
\setcounter{pro}{0}

As indicated in \cite[Proposition~2]{LV21},
a polynomial differential system has a global center if and only if
it has only one equilibrium and it is a center and
the system is {\it monodromic at infinity}
(i.e., no orbits connecting with any equilibrium on the equator of the Poincar\'e sphere).
So we start with monodromy at infinity and prove the following.

\begin{lm}
If system \eqref{equ:Newton} is monodromic at infinity, then $m\le 2$.
\label{lm:NtoC}
\end{lm}

In order to discuss orbits near infinity,
we need to compactify the phase plane $\mathbb{R}^2$ (\cite{DLA, ZZF}).
As shown in \cite[Chapter~5]{DLA},
one usually uses the Poincar\'e sphere (disk), the Poincar\'e-Lyapunov sphere (disk) or the Bendixson sphere for compactification,
but none of them is convenient to use in our case.
Using the Poincar\'e sphere for compactification,
system~\eqref{equ:Newton} takes the following two forms
\begin{align*}
\dot u=\sum_{i=0}^mP_i \left(\frac{1}{v}\right)u^iv^{d+1-i}-u^2v^d,~~~
\dot v=-uv^{d+1}
\end{align*}
and
\begin{align}
\dot u=v^d-u\sum_{i=0}^mP_i\left(\frac{u}{v}\right)v^{d+1-i},~~~
\dot v=-\sum_{i=0}^mP_i\left(\frac{u}{v}\right)v^{d+2-i}
\label{PC2}
\end{align}
separately in the two local charts of the equator of the Poincar\'e sphere,
where we made a time-rescaling ${\rm d}t\to v^d {\rm d}t$ and $d:=\max\{i+\deg P_i:i=0,...,m\}$.
Using the Poincar\'e-Lyapunov sphere for compactification,
we will get two similar systems to the above.
Using the Bendixson sphere for compactification,
system~\eqref{equ:Newton} near the infinity is transformed to the system
\begin{equation}
\left\{
\begin{array}{lllll}
\dot u=v(v^2-u^2)(u^2+v^2)^{d-1}
-2u\sum_{i=0}^mP_i\left(\frac{u}{u^2+v^2}\right)(u^2+v^2)^{d-i}v^{i+1},
\\
\dot v=(u^2-v^2)
\sum_{i=0}^mP_i\left(\frac{u}{u^2+v^2}\right)(u^2+v^2)^{d-i}v^i
-2uv^2(u^2+v^2)^{d-1}
\end{array}
\right.
\label{BC}
\end{equation}
near the origin.
Clearly, systems~\eqref{PC2} and \eqref{BC} are both complicated,
which makes difficulties to further desingularize degenerate equilibria at infinity.

We turn to use toroidal compactification, that is,
use the torus
$$
\mathbb{R}^2_\infty:=\big(\mathbb{R}\cup\{\infty\}\big)\times\big(\mathbb{R}\cup\{\infty\}\big)
$$
to compactify the phase plane $\mathbb{R}\times\mathbb{R}$,
as seen in Fig.~\ref{fig:TC}.
Alternatively,
we can use the rectangle
$
\big(\mathbb{R}\cup\{\pm\infty\}\big)\times\big(\mathbb{R}\cup\{\pm\infty\}\big)
$
to compactify the phase plane.
Under toroidal compactification,
the line $\mathbb{R}\times\{\infty\}$,
the line $\{\infty\}\times\mathbb{R}$ and
the point $\{\infty\}\times\{\infty\}$
of $\mathbb{R}^2_\infty$ represent the infinity of $\mathbb{R}^2$.
Thus we apply the following transformations
\begin{center}
{\bf (T1)} $y=1/v$, ~~~
{\bf (T2)} $x=1/u$, ~~~
{\bf (T3)} $x=1/u$ and $y=1/v$
\end{center}
to investigate orbits of system~\eqref{equ:Newton} near
the line $\mathbb{R}\times\{\infty\}$,
the line $\{\infty\}\times\mathbb{R}$ and
the point $\{\infty\}\times\{\infty\}$,
respectively.
Under {\bf (T1)}, {\bf (T2)} and {\bf (T3)},
system \eqref{equ:Newton} becomes the following three forms
\begin{align}
&\dot x=v^{\widetilde{m}-1},
&&\dot v=-\{P_m(x)+P_{m-1}(x)v+\cdots+P_0(x)v^m\big\}v^{\widetilde{m}-(m-2)},
\label{equ:T1}
\\
&\dot u=-u^{n+2}y,
&&\dot y=\widetilde{P}_0(u)+\widetilde{P}_1(u)y+\cdots+\widetilde{P}_m(u)y^m,
\label{equ:T2}
\\
&\dot u=u^{n+2}v^{\widetilde{m}-1},
&&\dot v=\{\widetilde{P}_m(u)+\widetilde{P}_{m-1}(u)+\cdots
+\widetilde{P}_0(u)v^m\}v^{\widetilde{m}-(m-2)},
\label{equ:T3}
\end{align}
respectively,
where a time-rescaling ${\rm d}t\to v^{m-2}{\rm d}t$ is maded in \eqref{equ:T1},
${\rm d}t\to u^n{\rm d}t$ in \eqref{equ:T2} and
${\rm d}t\to -u^nv^{\widetilde{m}}{\rm d}t$ in \eqref{equ:T3},
$n:=\max\{\deg P_0,...,\deg P_m\}$, $\widetilde{m}:=\max\{m-2,1\}$ and
\begin{align}
\widetilde{P}_i(u):=u^nP(1/u)~~~\mbox{for all}~i=0,...,m.
\end{align}
The first equation in each one of
the above three systems \eqref{equ:T1}-\eqref{equ:T3} is simply a monomial,
which makes it easier to further desingularization.

Under toroidal compactification,
system~\eqref{equ:Newton} is monodromic at infinity if and only if
the following three conditions hold:
{\bf(i)} system~\eqref{equ:T1} has no orbits
in the region $\{(x,v)\in\mathbb{R}^2:v\ne 0\}$
approaching to any point on the $x$-axis,
{\bf(ii)} system~\eqref{equ:T2} has no orbits
in the region $\{(u,y)\in\mathbb{R}^2:u\ne 0\}$
approaching to any point on the $y$-axis, and
{\bf(iii)} the equilibrium $(0,0)$ of system~\eqref{equ:T3} is a saddle.

\begin{figure}[h]
  \centering
  \includegraphics[height=1in]{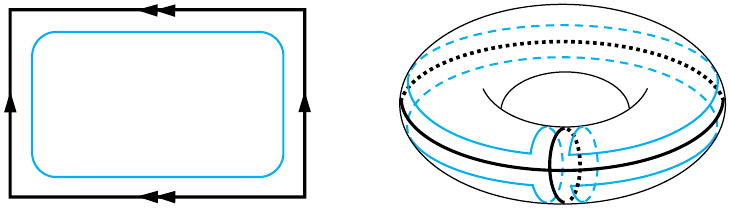}\\
  \caption{Rectangular compactification and toroidal compactification.}
  \label{fig:TC}
\end{figure}

{\bf Proof of Lemma~\ref{lm:NtoC}.}
If it is not true,
then $m\ge 3$ and system~\eqref{equ:T1} becomes
\begin{equation*}
\dot x=v^{m-3},
~~~
\dot v=-P_m(x)-P_{m-1}(x)v-\cdots-P_0(x)v^m.
\end{equation*}
Note that $P_m(x)\not\equiv 0$.
The orbit of the above system passing through any point $(x_*,0)$
such that $P_m(x_*)\ne 0$ intersects with the $x$-axis transversally
since $\dot v|_{(x_*,0)}=P_m(x_*)\ne 0$.
By the statement given just before the proof,
there is a contradiction to the monodromic assumption.
\qquad$\Box$

By Lemma~\ref{lm:NtoC},
we only need to consider the polynomial Cherkas system
\begin{align}
\dot x=y,~~~
\dot y=P_0(x)+P_1(x)y+P_2(x)y^2,
\label{equ:cherkas}
\end{align}
where
$P_0(x):=a_0+\cdots+a_nx^n$,
$P_1(x):=b_0+\cdots+b_nx^n$,
$P_2(x):=c_0+\cdots+c_nx^n$,
$n\ge 0$
and $a_n^2+b_n^2+c_n^2\ne 0$.
We assume that $P_0(x)\not\equiv 0$; otherwise,
system~\eqref{equ:cherkas} cannot be monodromic at infinity
since the $x$-axis is full of equilibria.
We also assume that $P_2(x)\not\equiv 0$
since the case $P_2(x)\equiv0$ has been studied in \cite{CLZ,GG15,LV22},
as indicated in the introduction.

In what follows,
we use quasi-homogeneous blow-ups associated with Newton polygons (\cite[Chapter~3]{DLA})
to characterize monodromy at infinity for Cherkas system~\eqref{equ:cherkas}.
Define the vector fields
\begin{align}
{\cal X}^{(0)}
:=F_1^{(0)}(u,v)\frac{\partial}{\partial u}+G_1^{(0)}(u,v)\frac{\partial}{\partial v},
~~~
{\cal Y}^{(0)}
:=F_2^{(0)}(u,y)\frac{\partial}{\partial u}+G_2^{(0)}(u,y)\frac{\partial}{\partial y},
\label{equ:X0Y0}
\end{align}
where $F_1^{(0)}(u,v):=u^{n+2}$, $G_1^{(0)}(u,v):=\widetilde{P}_0(u)v^3+\widetilde{P}_1(u)v^2+\widetilde{P}_2(u)v$,
\begin{align*}
&F_2^{(0)}(u,y):=-u^{n+2}(y+y_*),
&G_2^{(0)}(u,y):=\widetilde{P}_0(u)+\widetilde{P}_1(u)(y+y_*)+\widetilde{P}_2(u)(y+y_*)^2,
\end{align*}
and $y_*:={-b_n}/{(2c_n)}$ with nonzero $c_n$.
Suppose that $F_1^{(0)}$ and $G_1^{(0)}$ have expansions
\begin{align*}
F_1^{(0)}(u,v)=\sum_{i+j\ge 0} f_{i,j}^{(0)} u^i v^j
~~~\mbox{and}~~~
G_1^{(0)}(u,v)=\sum_{i+j\ge 0} g_{i,j}^{(0)} u^i v^j.
\end{align*}
As defined in \cite[p.104]{DLA},
the lattice point set
$$
{\cal S}({\cal X}^{(0)}):=\{(i,j)\in\mathbb{Z}^2:(f_{i+1,j}^{(0)},g_{i,j+1}^{(0)})\ne (0,0)\}
$$
is referred to as the {\it support} of the vector field ${\cal X}^{(0)}$.
Embed the set ${\cal S}({\cal X}^{(0)})$ in the $(u,v)$-plane and
consider the lower convex semi-hull
$$
{\rm conv}\{{\cal S}({\cal X}^{(0)})+\mathbb{R}^2_+\},
$$
whose boundary consists of a vertical line, a horizontal line and a compact polygon,
which may also be a single lattice point.
The compact polygon is called the {\it Newton polygon} of the vector field ${\cal X}^{(0)}$,
denoted by ${\cal N}({\cal X}^{(0)})$.
If the Newton polygon ${\cal N}({\cal X}^{(0)})$ has an edge $E^{(0)}$
lying on the line
$$
q^{(0)}u+p^{(0)}v=\sigma^{(0)}
$$
for a pair of coprime positive integers $p^{(0)}$ and $q^{(0)}$,
then we define the {\it polynomial corresponding to the edge $E^{(0)}$} as
\begin{align}
\mathcal{P}_{E^{(0)}}(v):=
\sum_{(i,j)\in E^{(0)}\cap \mathbb{Z}^2}
(q^{(0)} g_{i,j+1}^{(0)}-p^{(0)} f_{i+1,j}^{(0)})v^{j+1}.
\label{defLE}
\end{align}
Moreover,
the {\it height} and {\it width} of the edge $E^{(0)}$ are
the lengthes of the projections on the $v$-axis and $u$-axis,
denoted by $h(E^{(0)})$ and $\omega(E^{(0)})$, respectively.
If the polynomial $\mathcal{P}_{E^{(0)}}$ has a real root $\phi^{(0)}$,
then we blow up the degenerate equilibrium $(0,0)$ of the vector field ${\cal X}^{(0)}$
in the positive $u_1$-axis
by the quasi-homogeneous blow-up $u=u_1^{q^{(0)}}$ and $v=u_1^{p^{(0)}}(\phi^{(0)}+v_1)$
and obtain the vector field
\begin{align*}
{\cal X}^{(1)}={\cal D}({\cal X}^{(0)};p^{(0)},q^{(0)},\phi^{(0)})
:=F_1^{(1)}(u_1,v_1)\frac{\partial}{\partial u_1}
+G_1^{(1)}(u_1,v_1)\frac{\partial}{\partial v_1},
\end{align*}
where a common factor $u^{\sigma^{(0)}}/q^{(0)}$ is eliminated, and
{\small
\begin{align*}
F_1^{(1)}(u_1,v_1)
&:=\frac{F_1^{(0)}(u,v)}{u_1^{\sigma^{(0)}+q^{(0)}-1}}
=u_1\left\{\sum_{(i,j)\in E^{(0)}\cap \mathbb{Z}^2}
f_{i+1,j}^{(0)}(\phi^{(0)}+v_1)^j+O(u_1)\right\},
\\
G_1^{(1)}(u_1,v_1)
&:=\frac{q^{(0)}G_1^{(0)}(u,v)}{u_1^{\sigma^{(0)}+p^{(0)}}}
-\frac{p^{(0)}(\phi^{(0)}+v_1)F_1^{(0)}(u,v)}{u_1^{\sigma^{(0)}+q^{(0)}}}
=\mathcal{P}_{E^{(0)}}(\phi^{(0)}+v_1)+O(u_1).
\end{align*}
}It is similar to define the above notations and concepts for the vector field ${\cal Y}^{(0)}$,
defined in \eqref{equ:X0Y0}.

\begin{lm}
Cherkas system~\eqref{equ:cherkas} is monodromic at infinity if and only if
$n$ is odd and either
\begin{description}
  \item[(M1)]
  $c_n=b_n=0$, $a_n<0$, and
  there is an integer $i_*\le n-1$ such that
  the Newton polygon ${\cal N}({\cal X}^{(i)})$ has exactly one edge $E^{(i)}$,
  whose height is 2 and width is $2p^{(i)}$ for an integer $p^{(i)}$,
  and the polynomial $\mathcal{P}_{E^{(i)}}$ has the only nonzero real root $\phi^{(i)}$,
  which is of multiplicity 2, for all $i=0,...,i_*-1$,
  and ${\cal N}({\cal X}^{(i_*)})$ has exactly one edge $E^{(i_*)}$,
  whose height is $2$ and width is $2p^{(i_*)}$ for an integer $p^{(i_*)}$,
  and the polynomial $\mathcal{P}_{E^{(i_*)}}$ has no nonzero real roots,
  where ${\cal X}^{(i+1)}:={\cal D}({\cal X}^{(i)};p^{(i)},1,\phi^{(i)})$
  for all $i=0,...,i_*-1$ and the definition of height (and width) of an edge is given just below \eqref{defLE}; or

  \item[(M2)]
  $c_n<0$ and $b_n^2-4a_n c_n<0$; or

  \item[(M3)]
  $c_n<0$, $b_n^2-4a_n c_n=0$, and
  ${\cal Y}^{(0)}$ satisfies the same condition as ${\cal X}^{(0)}$
  given in {\bf(M1)} with integer $i_*\le n$ $($or $\le n+1$$)$
  if $b_n\ne 0$ $($or $b_n=0$$)$.
\end{description}
\label{lm:AIC}
\end{lm}

{\bf Proof.}
According to the statement given just before the proof of Lemma~\ref{lm:NtoC},
we need to consider systems~\eqref{equ:T1}-\eqref{equ:T3} with $m=2$
when dealing with monodromy at infinity for Cherkas system~\eqref{equ:cherkas}.
That is,
analyze phase portrait of the system
\begin{align}
\left\{
\begin{array}{llll}
\dot x=1,
\\
\dot v=- P_2(x) v- P_1(x) v^2-P_0(x) v^3
\end{array}
\right.
\label{PP1}
\end{align}
near the line $v=0$,
determine phase portrait of the system
\begin{align}
\left\{
\begin{array}{llll}
\dot u=-u^{n+2}y,
\\
\dot y=\widetilde{P}_0(u)+\widetilde{P}_1(u)y+\widetilde{P}_2(u)y^2
=a_n+b_ny+c_ny^2+O(u)
\end{array}
\right.
\label{PP2}
\end{align}
near the line $u=0$,
and investigate phase portrait of the system
\begin{align}
\left\{
\begin{array}{llll}
\dot u=u^{n+2},
\\
\dot v=\widetilde{P}_2(u)v+ \widetilde{P}_1(u) v^2+\widetilde{P}_0(u) v^3
=v\big(c_n+b_nv+a_nv^2+O(u)\big)
\end{array}
\right.
\label{PP3}
\end{align}
near the origin $(0,0)$.
Note that
system~\eqref{equ:cherkas} is monodromic at infinity
if and only if
the following three conditions hold:
{\bf(i)} system~\eqref{PP1} has no orbits in the region
$\{(x,v)\in\mathbb{R}^2:v\ne 0\}$
approaching to any point on the $x$-axis,
{\bf(ii)} system~\eqref{PP2} has no orbits in the region
$\{(u,y)\in\mathbb{R}^2:u\ne 0\}$
approaching to any point on the $y$-axis, and
{\bf(iii)} the equilibrium $(0,0)$ of system~\eqref{PP3} is a saddle.

We first consider the sufficiency.
If condition {\bf(M1)} holds,
it is clear that system~\eqref{PP1} has no orbits
in the region $\{(x,v)\in\mathbb{R}^2:v\ne 0\}$
approaching to any point on the $x$-axis.
Moreover,
since $a_n<0$ and $b_n=c_n=0$,
system~\eqref{PP2} has no equilibria on the $y$-axis and therefore
it has no orbits in the region $\{(u,y)\in\mathbb{R}^2:u\ne 0\}$
approaching to any point on the $y$-axis.
So we only need to show that
the degenerate equilibrium $(0,0)$ with zero linear part is a saddle of system~\eqref{PP3}.
For this purpose,
we use quasi-homogeneous blow-ups associated with Newton polygons to desingularization.
Note that the Newton polygon of
the vector field ${\cal X}^{(0)}$ generated by system~\eqref{PP3}
has exactly one edge $E^{(0)}$,
which links the vertex $(0,2)$ with the vertex $(2p^{(0)},0)$.
Applying the quasi-homogeneous blow-up
$u=w_1 z_1$ and $v=z_1^{p^{(0)}}$
to blow up the degenerate equilibrium $(0,0)$ of ${\cal X}^{(0)}$
in the positive $v$-direction,
we obtain the vector field
\begin{align}
{\cal X}_h^{(1)}
:=w_1\big(-g_{0,3}^{(0)}+o(1)\big)\frac{\partial }{\partial w_1}
+z_1\big(g_{0,3}^{(0)}+o(1)\big)\frac{\partial }{\partial z_1},
\label{vf:Xh1}
\end{align}
where a common factor $z_1^{2p^{(0)}}/p^{(0)}$ is eliminated
and $g_{0,3}^{(0)}$ ($=a_n$) is the coefficient of $G_1^{(0)}$
corresponding to the left end-point $(0,2)$ of the edge $E^{(0)}$.
The equilibrium $(0,0)$ of the vector field ${\cal X}_h^{(1)}$ is a hyperbolic saddle
whose stable manifold and unstable manifold lie on the axes.
It is similar to blow up the degenerate equilibrium $(0,0)$
of ${\cal X}^{(0)}$ in the negative $v$-direction
by the transformation $u=w_1 z_1$ and $v=-z_1^{p^{(0)}}$
and find that
the equilibrium $(0,0)$ of the transformed vector field is also a hyperbolic saddle
whose stable manifold and unstable manifold lie on the axes.
On the other hand,
we blow up in the $u$-direction by the transformation
$u=u_1$ and $v=u_1^{p^{(0)}}(\phi^{(0)}+v_1)$ and obtain the vector field
\begin{align}
{\cal X}^{(1)}
=u_1^{n+2-2p^{(0)}}\frac{\partial }{\partial u_1}
+\big(\mathcal{P}_{E^{(0)}}(\phi^{(0)}+v_1)+O(u_1)\big)\frac{\partial }{\partial v_1},
\label{vf:X1}
\end{align}
where a common factor $u_1^{2p^{(0)}}$ is eliminated and
$$
\mathcal{P}_{E^{(0)}}(\phi^{(0)}+v_1)=g_{0,3}^{(0)} v_1^2(\phi^{(0)}+v_1).
$$
The vector field ${\cal X}^{(1)}$ has two equilibria
$(0,0)$ and $(0,-\phi^{(0)})$ on the $v_1$-axis.
We see from the the property of ${\cal N}({\cal X}^{(1)})$ given in {\bf(M1)} that $n+2-2p^{(0)}\ge 2$.
Then the equilibrium $(0,-\phi^{(0)})$ is semi-hyperbolic,
i.e., the Jacobian matrix at this equilibrium has exactly one zero eigenvalue.
By \cite[Theorem~2.19]{DLA} or \cite[Theorem~7.1, p.114]{ZZF},
the equilibrium $(0,-\phi^{(0)})$ is a saddle since $n$ is odd and $g_{0,3}^{(0)}=a_n<0$.
For the degenerate equilibrium $(0,0)$,
we see from {\bf(M1)} that the Newton polygon ${\cal N}({\cal X}^{(1)})$
has exactly one edge $E^{(1)}$,
linking the vertex $(0,1)$ with the vertex $(2p^{(1)},-1)$.
Then blowing up the degenerate equilibrium $(0,0)$ of ${\cal X}^{(1)}$
in the positive $v_1$-direction by the transformation
$u_1=w_2 z_2$ and $v_1=z_2^{p^{(1)}}$,
we obtain the vector field
\begin{align}
{\cal X}_h^{(2)}
:=w_2\big(-g_{0,2}^{(1)}+o(1)\big)\frac{\partial }{\partial w_2}
+z_2\big(g_{0,2}^{(1)}+o(1)\big)\frac{\partial }{\partial z_2},
\label{vf:Xh2}
\end{align}
where a common factor $z_2^{p^{(1)}}/p^{(1)}$ is eliminated
and $g_{0,2}^{(1)}$ ($=a_n$) is the coefficient of $G_1^{(1)}$ corresponding to
the left end-point $(0,1)$ of the edge $E^{(1)}$.
The equilibrium $(0,0)$ of ${\cal X}_h^{(2)}$ is a hyperbolic saddle
whose stable manifold and unstable manifold lie on the axes.
It is similar to blow up in the negative $v_1$-direction
by the transformation $u_1=w_2 z_2$ and $v_1=-z_2^{p^{(1)}}$
and find that the equilibrium $(0,0)$ of the transformed system is a hyperbolic saddle
whose stable manifold and unstable manifold lie on the axes.
On the other hand,
we blow up in the $u_1$-direction by
the transformation $u_1=u_2$ and $v_1=u_2^{p^{(1)}}(\phi^{(1)}+v_2)$ and obtain
\begin{align}
{\cal X}^{(2)}=u_2^{n+2-2p^{(0)}-p^{(1)}}\frac{\partial }{\partial u_2}
+\big(\mathcal{P}_{E^{(1)}}(\phi^{(1)}+v_2)+O(u_2)\big)\frac{\partial }{\partial v_2},
\end{align}
where a common factor $u_2^{p^{(1)}}$ is eliminated and
$$
\mathcal{P}_{E^{(1)}}(\phi^{(1)}+v_2)=g_{0,2}^{(1)} v_2^2.
$$
Note that ${\cal X}^{(2)}$ has the only equilibrium $(0,0)$ on the $v_2$-axis,
which is degenerate with zero linear part.
Similarly to the above,
for $i=3,4,...,i_*$,
by the property of ${\cal N}({\cal X}^{(i-1)})$ given in {\bf(M1)},
blowing up the degenerate equilibrium $(0,0)$ of the vector field ${\cal X}^{(i-1)}$
in the positive $v_{i-1}$-axis
by the transformation $u_{i-1}=w_i z_i$ and $v_{i-1}=z_i^{p^{(i-1)}}$,
we obtain
\begin{align}
{\cal X}_h^{(i)}:=w_i\big(-g_{0,2}^{(i-1)}+o(1)\big)\frac{\partial }{\partial w_i}
+z_i\big(g_{0,2}^{(i-1)}+o(1)\big)\frac{\partial }{\partial z_i},
\label{vf:Xhk}
\end{align}
where a common factor $z_i^{p^{(i-1)}}/p^{(i-1)}$ is eliminated
and $g_{0,2}^{(i-1)}$ ($=a_n$) is the coefficient of $G_1^{(i-1)}$ corresponding to
the left end-point $(0,1)$ of the edge $E^{(i-1)}$.
The equilibrium $(0,0)$ of ${\cal X}_h^{(i)}$ is a hyperbolic saddle
whose stable manifold and unstable manifold lie on the axes.
It is similar to blow up in the negative $v_{i-1}$-direction
by the transformation $u_{i-1}=w_i z_i$ and $v_{i-1}=-z_i^{p^{(i-1)}}$
and find that the equilibrium $(0,0)$ of the transformed system is a hyperbolic saddle
whose stable manifold and unstable manifold lie on the axes.
Moreover,
blowing up the degenerate equilibrium $(0,0)$ of the vector field ${\cal X}^{(i-1)}$
in the $u_{i-1}$-direction by the transformation $u_{i-1}=u_i$ and $v_{i-1}=u_i^{p^{(i-1)}}v_i$,
we obtain
\begin{align}
{\cal X}^{(i)}
=u_i^{\varrho^{(i)}}\frac{\partial}{\partial u_i}
+\big(\mathcal{P}_{E^{(i-1)}}(\phi^{(i-1)}+v_i)+O(u_i)\big)\frac{\partial}{\partial v_i},
\label{vf:Xk}
\end{align}
where a common factor $u_i^{p^{(i-1)}}$ is eliminated,
$\varrho^{(i)}:=n+2-2p^{(0)}-\sum_{j=1}^{i-1} p^{(i)}$
and
$$
\mathcal{P}_{E^{(i-1)}}(\phi^{(i-1)}+v_i)=g_{0,2}^{(i-1)}v_i^2.
$$
Note that the vector field ${\cal X}^{(i)}$ has the only equilibrium $(0,0)$
(degenerate with zero linear part) on the $v_i$-axis.
One more quasi-homogeneous blow-up to ${\cal X}^{(i_*)}$ ends the desingularization process
of the degenerate equilibrium $(0,0)$ of ${\cal X}^{(0)}$.
Actually,
applying the transformation
$u_{i_*}=w_{i_*+1} z_{i_*+1}$ and $v_{i_*}=z_{i_*+1}^{p^{(i_*)}}$
to blow up the degenerate equilibrium $(0,0)$ of ${\cal X}^{(i_*)}$
in the positive $v_{i_*}$-direction,
we obtain
{\small
\begin{align}
{\cal X}_h^{(i_*+1)}
=w_{i_*+1}\big(-g_{0,2}^{(i_*)}+o(1)\big)\frac{\partial }{\partial w_{i_*+1}}
+z_{i_*+1}\big(g_{0,2}^{(i_*)}+o(1)\big)\frac{\partial }{\partial z_{i_*+1}},
\label{vf:Xhis1}
\end{align}
}where a common factor $z_{i^*+1}^{p^{(i_*)}}/p^{(i_*)}$ is eliminated
and $g_{0,2}^{(i_*)}$ ($=a_n$) is the coefficient of $G_1^{(i_*)}$
corresponding to the left end-point $(0,1)$ of the edge $E^{(i_*)}$.
Clearly, the equilibrium $(0,0)$ of ${\cal X}_h^{(i_*+1)}$ is a hyperbolic saddle.
It is similar to blow up the degenerate equilibrium $(0,0)$ of ${\cal X}^{(0)}$
in the negative $v_{i_*}$-direction and find that
the equilibrium $(0,0)$ of the transformed system is also a hyperbolic saddle.
Moreover,
blowing up in the $u_{i_*}$-direction by the transformation
$u_{i_*}=u_{i_*+1}$ and $v_{i_*}=u_{i_*+1}^{p^{(i_*)}}v_{i_*+1}$,
we obtain
\begin{align}
{\cal X}^{(i_*+1)}
=u_{i_*+1}^{\varrho^{(i_*+1)}}\frac{\partial}{\partial u_{i_*+1}}
+\big(\mathcal{P}_{E^{(i_*)}}(v_{i_*+1})+O(u_{i_*+1})\big)\frac{\partial}{\partial v_{i_*+1}},
\label{vf:is1}
\end{align}
where a common factor $u_{i_*+1}^{p^{(i_*)}}$ is eliminated.
Note that $\mathcal{P}_{E^{(i_*)}}$ is a polynomial of degree 2 and
has no nonzero real roots by {\bf (M1)}.
Then ${\cal X}^{(i_*+1)}$ has no equilibria on the $v_{i_*+1}$-axis
and therefore
the degenerate equilibrium $(0,0)$ of ${\cal X}^{(0)}$ is desingularized.
After blowing down,
we find that the equilibrium $(0,0)$ of system~\eqref{PP3} is a saddle,
as seen in Fig.~\ref{fig:DX0}.
Consequently,
we see from the characterization given just below \eqref{PP3} that
system~\eqref{equ:cherkas} is monodromic at infinity.

\begin{figure}[h]
  \centering
  \includegraphics[height=1.5in]{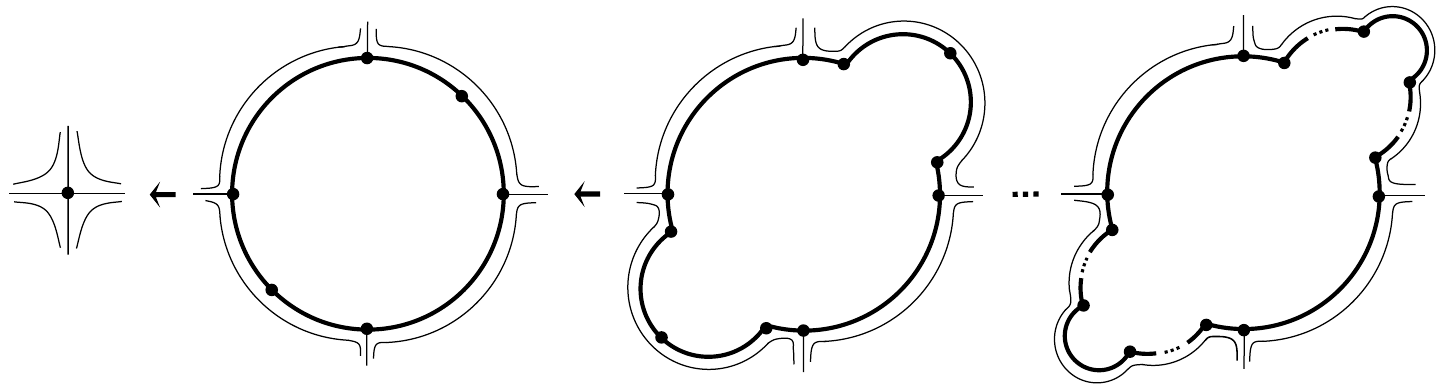}\\
  \caption{Desingularization process of ${\cal X}^{(0)}$.}\label{fig:des}
  \label{fig:DX0}
\end{figure}

If condition {\bf(M2)} holds,
it is clear that system~\eqref{PP1} has no orbits
in the region $\{(x,v)\in\mathbb{R}^2:v\ne 0\}$
approaching to any point on the $x$-axis.
Since $b_n^2-4a_n c_n<0$,
system~\eqref{PP2} has no equilibria on the $y$-axis and therefore
it has no orbits in the region $\{(u,y)\in\mathbb{R}^2:u\ne 0\}$
approaching to any point on the $y$-axis.
Moreover,
since $n$ is odd and $c_n<0$,
the equilibrium $(0,0)$ of system~\eqref{PP3} is a semi-hyperbolic saddle
by \cite[Theorem~2.19]{DLA} or \cite[Theorem~7.1, p.114]{ZZF}.
Thus, system~\eqref{equ:cherkas} is monodromic at infinity
because of the characterization given just below \eqref{PP3}.

If condition {\bf(M3)} holds,
it is clear that system~\eqref{PP1} has no orbits
in the region $\{(x,v)\in\mathbb{R}^2:v\ne 0\}$
approaching to any point on the $x$-axis.
Since $b_n^2-4a_n c_n=0$,
system~\eqref{PP2} has exactly one equilibrium $(0,y_*)$ on the $y$-axis,
degenerate with zero linear part,
where $y_*$ is defined just below \eqref{equ:X0Y0}.
Translating the equilibrium $(0,y_*)$ to the origin,
we obtain the corresponding vector field ${\cal Y}^{(0)}$.
Desingularizing ${\cal Y}^{(0)}$ similarly to ${\cal X}^{(0)}$ when condition {\bf(M1)} holds,
we obtain that system~\eqref{PP2} has no orbits in the region $\{(u,y)\in\mathbb{R}^2:u\ne 0\}$
approaching to the degenerate equilibrium $(0,y_*)$.
Therefore,
system~\eqref{PP2} has no orbits in the region $\{(u,y)\in\mathbb{R}^2:u\ne 0\}$
approaching to any point on the $y$-axis.
Moreover,
the same as case {\bf(M2)},
the equilibrium $(0,0)$ of system~\eqref{PP3} is a semi-hyperbolic saddle.
Thus, system~\eqref{equ:cherkas} is monodromic at infinity
and the sufficiency is proved.

In what follows, we consider the necessity.
Assume that system~\eqref{equ:cherkas} is monodromic at infinity
but none of conditions {\bf(M1)}, {\bf (M2)} and {\bf (M3)} holds.
Then there are seven cases:
\begin{equation*}
\begin{array}{lllll}
&\mbox{\bf(N1)}~n~\mbox{is even},
\\
&\mbox{\bf(N2)}~n~\mbox{is odd}~\mbox{and}~c_n>0,
\\
&\mbox{\bf(N3)}~n~\mbox{is odd}, c_n=0~\mbox{and}~b_n\ne 0,
\\
&\mbox{\bf(N4)}~n~\mbox{is odd}, c_n=b_n=0~\mbox{and}~a_n>0,
\\
&\mbox{\bf(N5)}~n~\mbox{is odd}, c_n=b_n=0,
a_n<0~\mbox{and}~{\cal X}^{(0)}~\mbox{does not satisfy {\bf(M1)}},
\\
&\mbox{\bf(N6)}~n~\mbox{is odd}, c_n< 0~\mbox{and}~b_n^2-4a_nc_n>0,
\\
&\mbox{\bf(N7)}~n~\mbox{is odd}, c_n< 0, b_n^2-4a_nc_n=0
~\mbox{and}~{\cal Y}^{(0)}~\mbox{does not satisfy {\bf(M3)}}.
\end{array}
\end{equation*}

In case {\bf(N1)},
on the invariant line $v=0$ of system~\eqref{PP3},
we have $\dot u=u^{n+2}$ with even $n$,
implying that the equilibrium $(0,0)$ of system~\eqref{PP3} is not a saddle.
We see from  the characterization given just below \eqref{PP3} that
system~\eqref{equ:cherkas} is not monodromic at infinity,
a contradiction.

In case {\bf(N2)},
by \cite[Theorem~2.19]{DLA} or \cite[Theorem~7.1, p.114]{ZZF},
the equilibrium $(0,0)$ of system~\eqref{PP3} is a semi-hyperbolic node and therefore
system~\eqref{equ:cherkas} is not monodromic at infinity,
a contradiction.

In case {\bf(N3)},
system~\eqref{PP2} has a unique equilibrium $(0,-a_n/b_n)$ on the $y$-axis,
which is semi-hyperbolic.
By \cite[Theorem~2.19]{DLA} or \cite[Theorem~7.1, p.114]{ZZF},
there are orbits in the region $\{(u,y)\in\mathbb{R}^2:u\ne 0\}$
approaching to the equilibrium $(0,-a_n/b_n)$.
We see from  the characterization given just below \eqref{PP3} that
system~\eqref{equ:cherkas} is not monodromic at infinity,
a contradiction.

In case {\bf(N4)},
we see from \eqref{PP3} that $\dot u=u^{n+2}$ with odd $n$ on the invariant line $v=0$
and $\dot v=a_nv^3$ with $a_n>0$ on the invariant line $u=0$.
It follows that the equilibrium $(0,0)$ of system~\eqref{PP3} is not a saddle
and therefore system~\eqref{equ:cherkas} is not monodromic at infinity,
a contradiction.


\begin{figure}[h]
    \centering
     \subcaptionbox{%
     }{\includegraphics[height=1in]{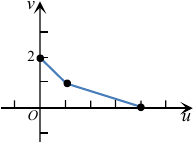}}~~~~~~
     \subcaptionbox{%
     }{\includegraphics[height=1in]{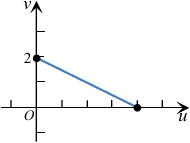}}
    \caption{Newton polygon ${\cal N}({\cal X}^{(0)})$.}
    \label{fig:NX10}
\end{figure}


In case {\bf(N5)},
we consider the vector field ${\cal X}^{(0)}$ generated by system~\eqref{PP3}.
Since $G_1^{(0)}(0,v)=a_nv^3$ and $v$ is a simple factor of $G_1^{(0)}(u,v)$,
the left end-point of ${\cal N}({\cal X}^{(0)})$ is the point $(0,2)$
and the right end-point lies on the $u$-axis.
It follows that the Newton polygon of ${\cal X}^{(0)}$ has either two edges or exactly one edge,
as illustrated by Fig.~\ref{fig:NX10}.
Since ${\cal X}^{(0)}$ does not satisfy condition {\bf(M1)},
the Newton polygon ${\cal N}({\cal X}^{(0)})$ has the following four subcases:
\begin{description}
  \item[(S1)] two edges; or

  \item[(S2)] one edge $E^{(0)}$, $h(E^{(0)})=2$ and $\omega(E^{(0)})$ is odd; or

  \item[(S3)] one edge $E^{(0)}$, $h(E^{(0)})=2$,
  $\omega(E^{(0)})=2p^{(0)}$ for an integer $p^{(0)}$,
  and $\mathcal{P}_{E^{(0)}}$ has a simple nonzero real root; or

  \item[(S4)] one edge $E^{(0)}$, $h(E^{(0)})=2$,
  $\omega(E^{(0)})=2p^{(0)}$ for an integer $p^{(0)}$,
  and $\mathcal{P}_{E^{(0)}}$ has a double nonzero real root $\phi^{(0)}$.
\end{description}

In subcase {\bf(S1)},
we assume that
the first edge $\tilde{E}^{(0)}$ of ${\cal N}({\cal X}^{(0)})$
links the vertex $(0,2)$ with the vertex $(\tilde{p}^{(0)},1)$ for an integer $\tilde{p}^{(0)}$.
Then
$$
\mathcal{P}_{\tilde{E}^{(0)}}(v)=g_{0,3}^{(0)} v^3+g_{\tilde{p}^{(0)},2}^{(0)}v^2,
$$
which has one (simple) nonzero real root $\tilde{\phi}^{(0)}$.
Applying the quasi-homogeneous blow-up
$u=u_1$ and $v=u_1^{\tilde{p}^{(0)}}v_1$
to the vector field ${\cal X}^{(0)}$,
we obtain the vector field
\begin{align*}
u_1^{n+2-2\tilde{p}^{(0)}}\frac{\partial}{\partial u_1}
+\{\mathcal{P}_{\tilde{E}^{(0)}}(v_1)+O(u_1)\}\frac{\partial}{\partial v_1},
\end{align*}
where a common factor is eliminated.
Since $\widetilde{P}_2(u)\not\equiv 0$ and $\deg\widetilde{P}_2\le n$,
the abscissa of the right end-point of ${\cal N}({\cal X}^{(0)})$ is not greater than $n$.
It follows that $2\tilde{p}^{(0)}<n$ and therefore $n+2-2\tilde{p}^{(0)}>2$.
So the above vector field has a semi-hyperbolic equilibrium
$(0,\tilde{\phi}^{(0)})$ on the $v_1$-axis.
By \cite[Theorem~2.19]{DLA} or \cite[Theorem~7.1, p.114]{ZZF},
there is a $C^\infty$ invariant curve (not the $v_1$-axis)
passing through the equilibrium $(0,\tilde{\phi}^{(0)})$
and therefore system~\eqref{PP3} has a $C^\infty$ invariant curve
\begin{align*}
v=\tilde{\phi}^{(0)}u^{\tilde{p}^{(0)}}+O(u^{\tilde{p}^{(0)}+1}).
\end{align*}
So the equilibrium $(0,0)$ of system~\eqref{PP3} is not a saddle.
We see from  the characterization given just below \eqref{PP3} that
system~\eqref{equ:cherkas} is not monodromic at infinity,
a contradiction.

In subcase {\bf(S2)},
we assume that the edge $E^{(0)}$ links the vertex $(0,2)$ with the vertex $(\hat{p}^{(0)},0)$
and $\hat{p}^{(0)}$ is odd.
Then
$$
\mathcal{P}_{E^{(0)}}(v)=g_{0,3}^{(0)} v^3+g_{\hat{p}^{(0)},2}^{(0)}v.
$$
Applying the blow-up $u=u_1^2$ and $v=u_1^{\hat{p}^{(0)}}v_1$ to
the vector field ${\cal X}^{(0)}$, we obtain
\begin{align}
\frac{1}{2}u_1^{2n+3-2\hat{p}^{(0)}}\frac{\partial }{\partial u_1}
+\big\{\mathcal{P}_{E^{(0)}}(v_1)+O(u_1)\big\}
\frac{\partial }{\partial v_1},
\label{equ:P1S1}
\end{align}
where a common factor is eliminated.
Similarly to the subcase {\bf(S1)},
we have $\hat{p}^{(0)}\le n$ and therefore $2n+3-2\hat{p}^{(0)}\ge 3$.
Note that $g^{(0)}_{0,3}=a_n<0$.
In the circumstance $g^{(0)}_{\hat{p}^{(0)},1}>0$,
vector field \eqref{equ:P1S1} has two semi-hyperbolic equilibria
$(0,\pm(-g^{(0)}_{\hat{p}^{(0)},1}/g^{(0)}_{0,3})^{1/2})$ on the $v_1$-axis.
By \cite[Theorem~2.19]{DLA} or \cite[Theorem~7.1, p.114]{ZZF},
vector field~\eqref{equ:P1S1} has $C^\infty$ invariant curves
$$
v_1=\pm(-g^{(0)}_{\hat{p}^{(0)},1}/g^{(0)}_{0,3})^{\frac{1}{2}}+O(u_1)
$$
passing through the two semi-hyperbolic equilibria separately.
It follows that system~\eqref{PP3} has invariant curves given by
\begin{align*}
v=|u|^{\hat{p}^{(0)}/2}
\big\{\pm(-g^{(0)}_{\hat{p}^{(0)},1}/g^{(0)}_{0,3})^{\frac{1}{2}}
+O(|u|^{\frac{1}{2}})\big\},~~~u\ge0,
\end{align*}
$C^\infty$ functions of $|u|^{1/2}$,
passing through the equilibrium $(0,0)$
in the half-plane $u\ge 0$.
Thus the equilibrium $(0,0)$ of system~\eqref{PP3} is not a saddle, a contradiction.
In the opposite circumstance $g^{(0)}_{\hat{p}^{(0)},1}<0$,
changing $u$ into $-u$ brings it to the former one.
Hence,
system~\eqref{PP3} has invariant curves given by
\begin{align*}
v=|u|^{\hat{p}^{(0)}/2}
\big\{\pm(g^{(0)}_{\hat{p}^{(0)},1}/g^{(0)}_{0,3})^{\frac{1}{2}}
+O(|u|^{\frac{1}{2}})\big\},~~~u\le 0,
\end{align*}
$C^\infty$ functions of $|u|^{1/2}$,
passing through the equilibrium $(0,0)$
in the half-plane $u\le 0$.
Thus the equilibrium $(0,0)$ of system~\eqref{PP3} is not a saddle
and therefore system~\eqref{equ:cherkas} is not monodromic at infinity,
a contradiction.

In subcase {\bf(S3)},
under the blow-up $u=u_1$ and $v=u_1^{p^{(0)}}v_1$,
the vector field ${\cal X}^{(0)}$ becomes
\begin{align*}
u_1^{n+2-2p^{(0)}}\frac{\partial}{\partial u_1}
+\{\mathcal{P}_{E^{(0)}}(v_1)+O(u_1)\}\frac{\partial}{\partial v_1},
\end{align*}
where a common factor is eliminated.
Since $\mathcal{P}_{E^{(0)}}$ has a simple nonzero real root,
still denoted by $\phi^{(0)}$,
the above system has a semi-hyperbolic equilibria $(0,\phi^{(0)})$ on the $v_1$-axis.
Similar to subcase {\bf(S1)},
system~\eqref{PP3} has a $C^\infty$ invariant curve
\begin{align*}
v=\phi^{(0)}u^{p^{(0)}}+O(u^{p^{(0)}+1}).
\end{align*}
Then the equilibrium $(0,0)$ of system~\eqref{PP3} is not a saddle
and therefore system~\eqref{equ:cherkas} is not monodromic at infinity,
a contradiction.


\begin{figure}[h]
    \centering
    \subcaptionbox{%
     }{\includegraphics[height=1in]{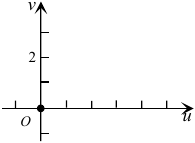}}~~~~~~
     \subcaptionbox{%
     }{\includegraphics[height=1in]{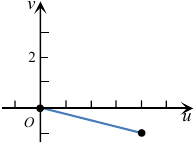}}
     \\
     \subcaptionbox{%
     }{\includegraphics[height=1in]{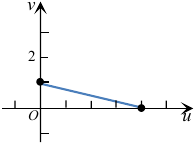}}~~~~~~
     \subcaptionbox{%
     }{\includegraphics[height=1in]{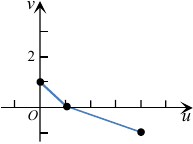}}~~~~~~
     \subcaptionbox{%
     }{\includegraphics[height=1in]{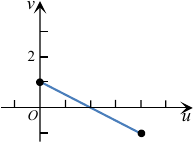}}
    \caption{Newton polygon ${\cal N}({\cal X}^{(i)})$ with $i \ge 1$.}
    \label{fig:NX11}
\end{figure}


In subcase {\bf(S4)},
applying the blow-up $u=u_1$ and $v=u_1^{p^{(0)}}(\phi^{(0)}+v_1)$
to the vector field ${\cal X}^{(0)}$, we obtain
\begin{align*}
{\cal X}^{(1)}
={\cal D}({\cal X}^{(0)};p^{(0)},1,\phi^{(0)})
=F_1^{(1)}(u_1,v_1)\frac{\partial}{\partial u_1}
+G_1^{(1)}(u_1,v_1)\frac{\partial}{\partial v_1},
\end{align*}
where a common factor $u_1^{2p^{(0)}}$ is eliminated and
$$
F_1^{(1)}(u_1,v_1)=u_1^{n+2-2p^{(0)}},~~~
G_1^{(1)}(u_1,v_1)=\mathcal{P}_{E^{(0)}}(\phi^{(0)}+v_1)+O(u_1).
$$
Note that $n+2-2p^{(0)}\ge 1$ and
$$
\mathcal{P}_{E^{(0)}}(\phi^{(0)}+v_1)=g_{0,3}^{(0)}v_1^2(\phi^{(0)}+v_1).
$$
Then ${\cal X}^{(1)}$ has support points $(n+2-2p^{(0)}-1,0)$ and $(0,1)$
and therefore the left end-point of ${\cal N}({\cal X}^{(1)})$ is
either the point $(0,0)$ if $n+2-2p^{(0)}=1$
or the point $(0,1)$ if $n+2-2p^{(0)}\ge 2$.
So, the Newton polygon ${\cal N}({\cal X}^{(1)})$ has 5 types:
the four {\bf(S1)}-{\bf(S4)} with $E^{(0)}$, $p^{(0)}$ and $\phi^{(0)}$
replaced with $E^{(1)}$, $p^{(1)}$ and $\phi^{(1)}$ respectively,
as illustrated by Fig.~\ref{fig:NX11}~(d) and (e),
and the one
\begin{description}
  \item[(S0)] either no edges, or one edge with height 1,
\end{description}
as illustrated by Fig.~\ref{fig:NX11}~(a), (b) and (c).

For {\bf(S0)},
if the left end-point of ${\cal N}({\cal X}^{(1)})$ is the point $(0,0)$,
then $F_1^{(1)}(u_1,v_1)=u_1$.
The Jacobian matrix at the equilibrium $(0,0)$ of ${\cal X}^{(1)}$ is given by
\begin{align*}
\left(
\begin{array}{lc}
1 & 0
\\
* & {\cal P}'_{E^{(0)}}(\phi^{(0)})
\end{array}
\right)
=
\left(
\begin{array}{lc}
1 & 0
\\
* & 0
\end{array}
\right).
\end{align*}
So the equilibrium $(0,0)$ is semi-hyperbolic
and there is an analytic invariant curve (not the $v_1$-axis)
passing through it by \cite[Theorem~2.19]{DLA} or \cite[Theorem~7.1, p.114]{ZZF}.
Hence the vector field ${\cal X}^{(0)}$ has an analytic invariant curve
$$
v=u^{p^{(0)}}(\phi^{(0)}+O(u)).
$$
On the contrary,
if the left end-point of ${\cal N}({\cal X}^{(1)})$ is the point $(0,1)$,
then ${\cal X}^{(1)}$ has the invariant line $v_1=0$ and therefore
${\cal X}^{(0)}$ has the invariant curve
$$
v=\phi^{(0)} u^{p^{(0)}}.
$$
Consequently,
the equilibrium $(0,0)$ of system~\eqref{PP3} is not a saddle and therefore
system~\eqref{equ:cherkas} is not monodromic at infinity,
a contradiction.

For {\bf(S1)},
the first edge $\tilde{E}^{(1)}$ of ${\cal N}({\cal X}^{(1)})$
links the vertex $(0,1)$ with the vertex $(\tilde{p}^{(1)},0)$ for an integer $\tilde{p}^{(1)}$.
Then
$$
\mathcal{P}_{\tilde{E}^{(1)}}(v_1)=g_{0,2}^{(1)} v_1^2+g_{\tilde{p}^{(1)},1}^{(1)}v_1,
$$
which has two (simple) real roots $0$ and
$\tilde{\phi}^{(1)}:=-g_{\tilde{p}^{(1)},1}^{(1)}/g_{0,2}^{(1)}$.
Applying the quasi-homogeneous blow-up
$u_1=u_2$ and $v_1=u_2^{\tilde{p}^{(1)}}v_2$ to ${\cal X}^{(1)}$,
we obtain
\begin{align*}
u_2^{n+2-2\tilde{p}^{(0)}-\tilde{p}^{(1)}}\frac{\partial}{\partial u_2}
+\{\mathcal{P}_{\tilde{E}^{(1)}}(v_2)+O(u_2)\}\frac{\partial}{\partial v_2},
\end{align*}
where a common factor is eliminated.
When $n+2-2\tilde{p}^{(0)}-\tilde{p}^{(1)}\ge 2$,
it is similar to ${\cal X}^{(0)}$ satisfying {\bf(S1)} that
the above vector field has a $C^\infty$ invariant curve (not the $v_1$-axis)
passing through the equilibrium $(0,\tilde{\phi}^{(1)})$.
After blowing down,
we see that system~\eqref{PP3} has a $C^\infty$ invariant curve of the form
\begin{align}
v=u^{p^{(0)}}(\phi^{(0)}+O(u)).
\label{vp0p0}
\end{align}
On the contrary,
when $n+2-2\tilde{p}^{(0)}-\tilde{p}^{(1)}=1$,
Jacobian matrices at the two equilibria $(0,0)$ and $(0,\tilde{\phi}^{(1)})$ are given by
\begin{align}
\left(
\begin{array}{lc}
1 & 0
\\
* & g_{\tilde{p}^{(1)},1}^{(1)}
\end{array}
\right)
~~~\mbox{and}~~~
\left(
\begin{array}{lc}
1 & 0
\\
* & -g_{\tilde{p}^{(1)},1}^{(1)}
\end{array}
\right),
\label{J-S-N}
\end{align}
respectively.
So one of the two equilibria is a hyperbolic saddle
and there is an analytic invariant curve (not the $v_2$-axis) passing through the saddle.
After blowing down,
we see that system~\eqref{PP3} has an analytic invariant curve of the form \eqref{vp0p0}.
As a consequence of the above two situations,
the equilibrium $(0,0)$ of system~\eqref{PP3} is not a saddle,
which leads to the same contradiction as above.

For {\bf(S2)} and {\bf(S3)},
we can prove similarly to the above subcase {\bf(S1)} that
${\cal X}^{(0)}$ has an invariant curve
$$
v=|u|^{p^{(0)}}(\phi^{(0)}+O(|u|^{\frac{1}{2}})),
$$
a $C^\infty$ function in $|u|^{\frac{1}{2}}$,
which leads to the same contradiction as above.

For {\bf(S4)},
we claim that
${\cal X}^{(i)}$ cannot satisfy {\bf(S4)}
with $E^{(0)}$, $p^{(0)}$ and $\phi^{(0)}$
replaced by $E^{(i)}$, $p^{(i)}$ and $\phi^{(i)}$, respectively,
for all $i=0,...,n$.
Otherwise,
note that
\begin{align}
{\cal X}^{(n)}=
u_n^{\varrho^{(n)}}
\frac{\partial}{\partial u_n}
+(\mathcal{P}_{E^{(n-1)}}(\phi^{(n-1)}+v_n)+O(u_n))
\frac{\partial}{\partial v_n},
\label{vf:Xi}
\end{align}
where $\varrho^{(n)}=n+2-2p^{(0)}-\sum_{i=1}^{n-1}p^{(i)}$.
Since $p^{(0)},...,p^{(n-1)}\ge 1$ and
$\varrho^{(n)}\ge 1$,
we have
$$
\varrho^{(n)}=1.
$$
Then the point $(0,0)$ is the left end-point of ${\cal N}({\cal X}^{(n)})$,
implying that ${\cal X}^{(n)}$ does not satisfy {\bf(S4)},
a contradiction.
Thus our claim given just before \eqref{vf:Xi} is proved.
The above analysis also suggests that the integer $i_*$ given in {\bf(M1)} satisfies that
$$
i_*\le n-1.
$$
By the claim,
if ${\cal X}^{(0)}$ does not satisfy {\bf(M1)},
then there is an integer $r\le n$ such that ${\cal X}^{(i)}$ satisfies {\bf(S4)}
with $E^{(0)}$, $p^{(0)}$ and $\phi^{(0)}$
replaced by $E^{(i)}$, $p^{(i)}$ and $\phi^{(i)}$ respectively,
for all $i=0,...,r-1$,
but ${\cal X}^{(r)}$ satisfies
one of conditions {\bf(S0)}-{\bf(S3)} with $E^{(0)}$ and $p^{(0)}$
replaced by $E^{(r)}$ and $p^{(r)}$, respectively, if necessary.
Similar to ${\cal X}^{(1)}$ satisfying one of {\bf(S0)}-{\bf(S3)} correspondingly,
${\cal X}^{(r)}$ has an invariant curve,
given by a $C^\infty$ function in $|u_r|^{1/2}$,
passing through the origin.
After blowing down,
${\cal X}^{(0)}$ has an invariant curve,
given by a nonzero $C^\infty$ function in $|u|^{1/2}$,
passing through the origin,
which leads to the same contradiction as above.

In case {\bf(N6)},
since $b_n^2-4a_nc_n> 0$,
system~\eqref{PP2} has two equilibria on the $y$-axis,
which are both semi-hyperbolic.
Similar to case {\bf(N3)},
system~\eqref{PP2} has orbits approaching to the two semi-hyperbolic equilibria
separately by \cite[Theorem~2.19]{DLA} or \cite[Theorem~7.1, p.114]{ZZF}.
Then system~\eqref{equ:cherkas} is not monodromic at infinity,
a contradiction.

In case {\bf(N7)},
we consider situations $b_n\ne 0$ and $b_n=0$ separately.
For the first situation,
we claim that ${\cal Y}^{(i)}$ can not satisfy {\bf(S4)}
with $E^{(0)}$, $p^{(0)}$ and $\phi^{(0)}$
replaced by $E^{(i)}$, $p^{(i)}$ and $\phi^{(i)}$, respectively,
for all $i=0,...,n+1$.
Otherwise,
{\small
\begin{align*}
{\cal Y}^{(n+1)}=
u_{n+1}^{\vartheta_0^{(n+1)}}U_{n+1}(u_{n+1},v_{n+1})\frac{\partial}{\partial u_{n+1}}
+\big(\mathcal{P}_{E^{(n)}}(\phi^{(n)}+v_{n+1})+O(u_{n+1})\big)
\frac{\partial}{\partial v_{n+1}},
\end{align*}
}where $\vartheta_0^{(n+1)}:=n+2-\sum_{i=0}^{n}p^{(i)}$,
$$
U_{n+1}(u_{n+1},v_{n+1}):=y_*+\phi^{(0)}u_{n+1}^{p^{(0)}}+O(u_{n+1}^{p^{(0)}+1}),
$$
and $y_*:=-b_n/(2c_n)\ne 0$.
Since $p^{(0)},...,p^{(n)}\ge 1$ and
$\vartheta_0^{(n+1)}\ge 1$,
we have
$$
\vartheta_0^{(n+1)}=1.
$$
Then the point $(0,0)$ is the left end-point of ${\cal N}({\cal Y}^{(n+1)})$,
implying that ${\cal Y}^{(n+1)}$ does not satisfy {\bf(S4)}, a contradiction.
This proves the above claim.
The above analysis also suggests that
in the situation $b_n\ne 0$
the integer $i_*$ given in {\bf(M3)} satisfies that
$$
i_*\le n.
$$
Moreover, by the above claim,
if ${\cal Y}^{(0)}$ does not satisfy {\bf(M3)},
there is an integer $r'\le n+1$ such that ${\cal Y}^{(i)}$ satisfies {\bf(S4)}
with $E^{(0)}$, $p^{(0)}$ and $\phi^{(0)}$
replaced by $E^{(i)}$, $p^{(i)}$ and $\phi^{(i)}$, respectively,
for all $i=0,...,r'-1$,
but ${\cal Y}^{(r')}$ satisfies
one of conditions {\bf(S0)}-{\bf(S3)} with $E^{(0)}$ and $p^{(0)}$
replaced by $E^{(r')}$ and $p^{(r')}$, respectively, if necessary.
Similar to case {\bf(N5)},
the vector field ${\cal Y}^{(0)}$ has an invariant curve,
given by a $C^\infty$ function in $|u|^{1/2}$,
passing through the origin.
Hence system~\eqref{equ:cherkas} is not monodromic at infinity,
a contradiction.

In the situation that $b_n=0$,
we claim that ${\cal Y}^{(i)}$ cannot satisfy {\bf(S4)}
with $E^{(0)}$, $p^{(0)}$ and $\phi^{(0)}$
replaced by $E^{(i)}$, $p^{(i)}$ and $\phi^{(i)}$, respectively,
for all $i=0,...,n+2$.
Otherwise,
since $y_*=0$, we have
{\small
\begin{align*}
{\cal Y}^{(n+2)}=
u_{n+2}^{\tau^{(n+2)}}\big(\phi^{(0)}+O(u_{n+2})\big)
\frac{\partial}{\partial u_{n+2}}
+\big(\mathcal{P}_{E^{(n+1)}}(\phi^{(n+1)}+v_{n+2})+O(u_{n+2})\big)
\frac{\partial}{\partial v_{n+2}},
\end{align*}
}where $\tau^{(n+2)}:=n+2-\sum_{i=1}^{n+1}p^{(i)}$.
Since $p^{(1)},...,p^{(n+1)}\ge 1$ and
$\tau^{(n+2)}\ge 1$,
we have
$$
\tau^{(n+2)}=1.
$$
Then the point $(0,0)$ is the left end-point of ${\cal N}({\cal Y}^{(n+2)})$,
implying that ${\cal Y}^{(n+2)}$ does not satisfy {\bf(S4)}, a contradiction.
Thus the above claim is proved.
The above analysis also suggests that
in the situation $b_n=0$ the integer $i_*$ given in {\bf(M3)} satisfies that
$$
i_*\le n+1.
$$
Moreover, by the above claim,
if ${\cal Y}^{(0)}$ does not satisfy {\bf(M3)},
then there is an integer $r''\le n+2$ such that ${\cal Y}^{(i)}$ satisfies {\bf(S4)}
with $E^{(0)}$, $p^{(0)}$ and $\phi^{(0)}$
replaced by $E^{(i)}$, $p^{(i)}$ and $\phi^{(i)}$, respectively,
for all $i=0,...,r''-1$,
but ${\cal Y}^{(r'')}$ satisfies
one of conditions {\bf(S0)}-{\bf(S3)} with $E^{(0)}$ and $p^{(0)}$
replaced by $E^{(r'')}$ and $p^{(r'')}$, respectively, if necessary.
Similar to case {\bf(N5)},
the vector field ${\cal Y}^{(0)}$ has an invariant curve,
given by a $C^\infty$ function in $|u|^{1/2}$,
passing through the origin,
and therefore system~\eqref{equ:cherkas} is not monodromic at infinity,
a contradiction.

As a consequence of the above discussion in cases {\bf (N1)}-{\bf (N7)},
we obtain the necessity of Lemma~\ref{lm:AIC} and the proof is completed.
\qquad$\Box$

Further, we obtain from Lemma~\ref{lm:AIC} a connection between the monodromy at infinity
and the non-existence of some fractional formal invariant curves
of vector fields ${\cal X}^{(0)}$ and ${\cal Y}^{(0)}$ defined in\eqref{equ:X0Y0},
which provides a straightforward computational method to
characterize the monodromy at infinity.
Let $\mathbb{R}[\![|u|^{\frac{1}{2}}]\!]$
be the ring of formal power series in $|u|^{\frac{1}{2}}$
with coefficients in $\mathbb{R}$.
We call $C^{(0)}(u,v):=v-\Phi(u)=0$ a {\it $\frac{1}{2}$-fractional formal invariant curve}
of the vector field ${\cal X}^{(0)}$
if $\Phi(u)\in\mathbb{R}[\![|u|^{\frac{1}{2}}]\!]$ and
${\cal X}^{(0)}{\cal C}^{(0)}|_{{\cal C}^{(0)}=0}=0.$

\begin{thm}
Cherkas system~\eqref{equ:cherkas} is monodromic at infinity if and only if
$n$ is odd and either
\begin{description}
  \item[(F1)]
  $c_n=b_n=0$, $a_n<0$,
  and the vector field ${\cal X}^{(0)}$ has no formal invariant curves of the form
  $v=\Phi(u)\in\mathbb{R}[\![|u|^{\frac{1}{2}}]\!]\backslash\{0\}$ with $\Phi(0)=0$; or

  \item[(F2)]
  $c_n<0$ and $b_n^2-4a_n c_n<0$; or

  \item[(F3)]
  $c_n<0$, $b_n^2-4a_n c_n=0$,
  and the vector field ${\cal Y}^{(0)}$ has no formal invariant curves of the form
  $y=\Psi(u)\in\mathbb{R}[\![|u|^{\frac{1}{2}}]\!]$ with $\Psi(0)=0$.
\end{description}
\label{thm:AIC}
\end{thm}

{\bf Proof.}
It suffices to prove that
when $n$ is odd
conditions {\bf(F1)} and {\bf (F3)} are equivalent to
conditions {\bf (M1)} and {\bf (M3)}, respectively,
given in Lemma~\ref{lm:AIC}.
We first show that {\bf (M1)} implies {\bf (F1)}.
If {\bf (F1)} is invalid,
then ${\cal X}^{(0)}$ has a real formal invariant curve
$
{\cal C}^{(0)}(u,v):=v-\Phi^{(0)}(u)=0
$
defined for $u\ge 0$ or $u\le 0$,
where $\Phi^{(0)}(u)\in\mathbb{R}[\![|u|^{\frac{1}{2}}]\!]\setminus\{0\}$.
We only consider the case $u\ge 0$ since the case $u\le 0$ is similar.
Assume that
\begin{align}
\Phi^{(0)}(u)=\phi_0 u^{\iota^{(0)}_0}
+\sum_{k=1}^{s} \phi_k u^{\iota^{(0)}_k},
\label{Phi00}
\end{align}
where $s\in\mathbb{N}\cup\{+\infty\}$,
$0<\iota^{(0)}_0<\iota^{(0)}_1<\cdots<\iota^{(0)}_s$,
$2\iota^{(0)}_k \in \mathbb{N}$ and $\phi_k\ne 0$ for all $k\ge 0$.
The invariance suggests that
$$
0={\cal X}^{(0)}{\cal C}^{(0)}|_{{\cal C}^{(0)}=0}=
\left.
\left\{-F_1^{(0)}(u,v)\frac{d\Phi^{(0)}(u)}{du}+G_1^{(0)}(u,v)\right\}
\right|_{v=\Phi^{(0)}(u)}.
$$
Since {\bf (M1)} holds,
the Newton polygon ${\cal N}({\cal X}^{(0)})$ has exactly one edge $E^{(0)}$,
linking the vertex $(0,2)$ with the vertex $(2p^{(0)},0)$.
Then
\begin{align}
0={\cal X}^{(0)}{\cal C}^{(0)}|_{{\cal C}^{(0)}=0}
=\left\{
\begin{array}{lllll}
g^{(0)}_{0,3} \phi_0^3 u^{3\iota^{(0)}_0}+o(u^{3\iota^{(0)}_0})
&\mbox{if}~\iota^{(0)}_0<p^{(0)},
\\
\mathcal{P}_{E^{(0)}}(\phi_0)u^{3\iota^{(0)}_0}+o(u^{3\iota^{(0)}_0})
&\mbox{if}~\iota^{(0)}_0=p^{(0)},
\\
g^{(0)}_{2p^{(0)},1}\phi_0u^{2p^{(0)}+\iota^{(0)}_0}+o(u^{2p^{(0)}+\iota^{(0)}_0})
&\mbox{if}~\iota^{(0)}_0>p^{(0)},
\end{array}
\right.
\label{0XCC0}
\end{align}
which implies that $\iota^{(0)}_0=p^{(0)}$ and $\mathcal{P}_{E^{(0)}}(\phi_0)=0$.

In the case $i_*=0$ in {\bf (M1)},
we obtain a contradiction
since $\mathcal{P}_{E^{(i_*)}}=\mathcal{P}_{E^{(0)}}$ has no nonzero real roots.
Then {\bf (M1)} implies {\bf (F1)} in the case $i_*=0$.

In the oppositive case $i_*\ge 1$ in {\bf (M1)},
we have $\phi_0=\phi^{(0)}$,
the only nonzero real root of $\mathcal{P}_{E^{(0)}}$.
Further,
under the blow-up transformation $u=u_1$ and $v=u_1^{p^{(0)}}(\phi^{(0)}+v_1)$,
the vector field ${\cal X}^{(0)}$ becomes the vector field
${\cal X}^{(1)}={\cal D}({\cal X}^{(0)};p^{(0)},1,\phi^{(0)})$.
Correspondingly,
this transformation brings the formal curve ${\cal C}^{(0)}(u,v)=0$
into the formal curve
$
{\cal C}^{(1)}(u_1,v_1):=v_1-\Phi^{(1)}(u_1)=0,
$
where
$$
\Phi^{(1)}(u_1):=\frac{\Phi^{(0)}(u_1)}{u_1^{p^{(0)}}}-\phi_0
=\phi_1 u_1^{\iota^{(1)}_1}
+\sum_{k=2}^{s} \phi_k u_1^{\iota^{(1)}_k}
\in\mathbb{R}[\![|u_1|^{\frac{1}{2}}]\!]
$$
with $\iota^{(1)}_k:=\iota^{(0)}_k-p^{(0)}$ for all $k\ge 1$.
Since the formal curve ${\cal C}^{(0)}=0$ is invariant
under the vector field ${\cal X}^{(0)}$,
the formal curve ${\cal C}^{(1)}=0$ is correspondingly invariant
under the vector field ${\cal X}^{(1)}$.
Repeating the above process on ${\cal X}^{(0)}$ and ${\cal C}^{(0)}=0$ for $i\ge 1$,
we obtain from the invariance that
\begin{align}
0={\cal X}^{(i)}{\cal C}^{(i)}|_{{\cal C}^{(i)}=0}
=\left\{
\begin{array}{lllll}
g^{(i)}_{0,2} \phi_i^2 u_i^{2\iota^{(i)}_i}+o(u_i^{2\iota^{(i)}_i})
&\mbox{if}~\iota^{(i)}_i<p^{(i)},
\\
\mathcal{P}_{E^{(i)}}(\phi_i)u_i^{2\iota^{(i)}_i}+o(u_i^{2\iota^{(i)}_i})
&\mbox{if}~\iota^{(i)}_i=p^{(i)},
\\
g^{(i)}_{2p^{(i)},0} u_i^{2p^{(i)}}+o(u_i^{2p^{(i)}})
&\mbox{if}~\iota^{(i)}_i>p^{(i)},
\end{array}
\right.
\label{0XCCi}
\end{align}
where
${\cal X}^{(i)}={\cal D}({\cal X}^{(i-1)};p^{(i-1)},1,\phi^{(i-1)})$,
$E^{(i)}$ is the only edge of ${\cal N}({\cal X}^{(i)})$,
linking the vertex $(0,1)$ with the vertex $(2p^{(i)},-1)$,
and ${\cal C}^{(i)}(u_i,v_i):=v_i-\Phi^{(i)}(u_i)$ with
$$
\Phi^{(i)}(u_i):=\frac{\Phi^{(i-1)}(u_i)}{u_i^{p^{(i-1)}}}-\phi_{i-1}
=\phi_i u_i^{\iota^{(i)}_i}
+\sum_{k=i+1}^{s} \phi_{k} u_i^{\iota^{(i)}_{k}}
\in\mathbb{R}[\![|u_i|^{\frac{1}{2}}]\!]
$$
and $\iota^{(i)}_k:=\iota^{(i-1)}_k-p^{(i-1)}$ for all $k\ge i$.
It follows from \eqref{0XCCi} that
$\iota^{(i)}_i=p^{(i)}$ and $\mathcal{P}_{E^{(i)}}(\phi_i)=0$,
i.e., $\phi_i=\phi^{(i)}$, the only nonzero real root of $\mathcal{P}_{E^{(i)}}$.
Then the above process can be repeated for all $i=1,2,...,\max\{s,i_*\}$.

For the situation $s< i_*$,
we have ${\cal C}^{(s+1)}(u_{s+1},v_{s+1})=v_{s+1}-\Phi^{(s+1)}(u_{s+1})=v_{s+1}$
and therefore
\begin{align}
{\cal X}^{(s+1)}{\cal C}^{(s+1)}|_{{\cal C}^{(s+1)}=0}=G_1^{(s+1)}(u_{s+1},0)
=g_{2p^{(s+1)},0}^{(s+1)}u_{s+1}^{2p^{(s+1)}}+O(u_{s+1}^{2p^{(s+1)}+1})\ne 0,
\label{s<i*}
\end{align}
a contradiction,
where $g_{2p^{(s+1)},0}^{(s+1)}$ is the coefficient of $G_1^{(s+1)}$
corresponding the right end-point $(2p^{(s+1)},-1)$ of the edge $E^{(s+1)}$.

For the opposite situation $s\ge i_*$,
since the polynomial $\mathcal{P}_{E^{(i_*)}}$ has no nonzero real roots,
equation \eqref{0XCCi} can not hold for $i=i_*$,
a contradiction.
Consequently, {\bf (M1)} implies {\bf (F1)}.

On the contrary,
we show that {\bf (F1)} implies {\bf (M1)}.
For a reduction to absurdity,
we assume that {\bf (M1)} does not hold.
By the claim given just before \eqref{vf:Xi},
there is an integer $r\le n$ such that
${\cal X}^{(i)}$ satisfies {\bf(S4)}
with $E^{(0)}$, $p^{(0)}$ and $\phi^{(0)}$
replaced by $E^{(i)}$, $p^{(i)}$ and $\phi^{(i)}$, respectively,
for all $i=0,...,r-1$,
but ${\cal X}^{(r)}$ satisfies one of {\bf(S0)}-{\bf(S3)}
with $E^{(0)}$ and $p^{(0)}$
replaced by $E^{(r)}$ and $p^{(r)}$, respectively, if necessary.
It is shown in case {\bf(N5)} of the proof of Lemma~\ref{lm:AIC} that
${\cal X}^{(r)}$ has an invariant curve,
given by a $C^\infty$ function in $|u_{r}|^{1/2}$,
passing through the origin.
Hence, ${\cal X}^{(r)}$ has a $\frac{1}{2}$-fractional formal invariant curve
$$
v_{r}=\Phi_{r}(u_{r})\in\mathbb{R}[\![|u_{r}|^{\frac{1}{2}}]\!]
$$
with $\Phi_{r}(0)=0$.
After blowing down,
${\cal X}^{(0)}$ has a $\frac{1}{2}$-fractional formal invariant curve
$$
v=\Phi(u)=|u|^{p^{(0)}}\big(\phi^{(0)}+O(|u|^{\frac{1}{2}})\big)
\in\mathbb{R}[\![|u|^{\frac{1}{2}}]\!]\backslash\{0\}
$$
a contradiction to {\bf(F1)}.
Consequently,
we obtain the equivalence between {\bf(M1)} and {\bf(F1)}.
It is similar to show the equivalence between {\bf(M3)} and {\bf(F3)}
when $n$ is odd.
Thus the proof of this theorem is completed.
\qquad$\Box$

\begin{rmk}
{\rm
If Cherkas system~\eqref{equ:cherkas} is not monodromic at infinity,
then ${\cal X}^{(0)}$ would have an invariant curve $v=\Phi(u)\notin\mathbb{R}[\![|u|^{\frac{1}{2}}]\!]\backslash\{0\}$ with $\Phi(0)=0$.
This happens when a hyperbolic node appears in the desingularization process.
For example,
if $\lambda:=g_{\tilde{p}^{(1)},1}^{(1)}$ in \eqref{J-S-N} is irrational,
then the equilibrium $(0,0)$ of ${\cal X}^{(1)}$ is a non-resonant node.
By the $C^\infty$ normal form result given in \cite[Section~2.7]{DLA},
${\cal X}^{(1)}$ is of the form
$x\frac{\partial}{\partial x}+\lambda y\frac{\partial}{\partial y}$
under some $C^\infty$ coordinates
and has invariant curves $y=c x^{\lambda}$ for all $c\in \mathbb{R}$.
This suggests the existence of the above invariant curve of ${\cal X}^{(0)}$.
More complicated invariant curves appear in the resonant node case.
}
\label{rmk:node-ln}
\end{rmk}

\begin{rmk}
{\rm
When applying {\bf(F1)} of Theorem~\ref{thm:AIC},
we only need to compute $\frac{1}{2}$-fractional formal invariant curves
by the method of undetermined coefficients up to order $n$.
Actually,
consider a $\frac{1}{2}$-fractional formal curve $v=\Phi^{(0)}(u)$ of the form \eqref{Phi00}.
If {\bf(M1)} holds,
we see from \eqref{0XCC0}-\eqref{s<i*} that $s\ge i_*$,
\begin{align*}
\iota_k^{(0)}=p^{(0)}+\cdots+p^{(k)}~~~\mbox{and}~~~{\cal P}_{E^{(k)}}(\phi_k)=0
~~~\mbox{for all}~k=0,...,i_*.
\end{align*}
The equation ${\cal P}_{E^{(k)}}(\phi_k)=0$ has the solution
$\phi_k=\phi^{(k)}$ for all $k=0,...,i_*-1$,
however, the equation ${\cal P}_{E^{(i_*)}}(\phi_k)=0$ has no nonzero real solutions.
Moreover,
we see from expression \eqref{vf:Xhis1} of ${\cal X}^{(i_*+1)}$ that
$$
\varrho^{(i_*+1)}=n+2-p^{(0)}-\sum_{j=0}^{i_*} p^{(i)}\ge 1
$$
and therefore $\iota_{i_*}^{(0)}\le n+1-p^{(0)}\le n$.
Thus the above claimed is proved.
Similarly,
when applying {\bf(F3)} of Theorem~\ref{thm:AIC},
we only need to compute up to order $n+1$ (or $n+2$) for $b_n\ne 0$ (or $b_n=0$).
}
\end{rmk}

\section{Global center}
\setcounter{equation}{0}
\setcounter{lm}{0}
\setcounter{thm}{0}
\setcounter{rmk}{0}
\setcounter{df}{0}
\setcounter{cor}{0}
\setcounter{pro}{0}

By Lemma~\ref{lm:NtoC},
if system~\eqref{equ:Newton} has a global center at the origin,
then it is of the form \eqref{equ:cherkas}.
Since $P_0\ne 0$ as assumed just below \eqref{equ:cherkas},
if $P_1\ne 0$ then we assume that
$$
P_0(x)=a_{\iota_0}x^{\iota_0}+O(x^{\iota_0+1})~~~\mbox{and}~~~
P_1(x)=b_{\iota_1}x^{\iota_1}+O(x^{\iota_1+1})
$$
with nonzero coefficients $a_{\iota_0}$ and $b_{\iota_1}$
and integers $\iota_0$ and $\iota_1$.
By \cite[Theorem~3.5]{DLA},
system~\eqref{equ:cherkas} is monodromic at the origin
if and only if $\iota_0=2\nu-1$ for an integer $\nu\ge 1$ and moreover either
one of the following two conditions holds:
\begin{description}
  \item[(W1)]
  $\nu=1$, $a_{\iota_0}<0$ and $P_1(0)=0$; or

  \item[(W2)]
  $\nu\ge 2$, $a_{\iota_0}<0$,
   and either $\iota_1>\nu$,
   or $\iota_1=\nu$ and $b_{\iota_1}^2+4\nu a_{\iota_0}<0$,
   or $P_1=0$,
\end{description}
which are devoted to the non-degenerate case and the nilpotent degenerate case, respectively.
Further,
\cite[Theorem~3]{CS08} gives a complete characterization of system~\eqref{equ:cherkas}
having a local center at the origin.

\begin{thm}
Cherkas system~\eqref{equ:cherkas} satisfying either {\bf(W1)} or {\bf(W2)}
has a center at the origin if and only if
one of the following $($possibly overlapping$)$ conditions is satisfied:
\begin{description}
\item[(C1)]
The system has a symmetry in the $x$-axis, i.e., $P_1=0$;

\item[(C2)]
$P_i(x)=A_i(r(x))r'(x)$ for all $i=0,1,2$,
where $A_i(x)$\,s and $r(x)$ are all polynomials such that $A_0(0)<0$ and
$r(x)=x^{2\nu}+O(x^{2\nu+1})$;

\item[(C3)]
There is a local first integral of Darboux type,
i.e., first integral of the form
$\exp(f/g)\prod h_i^{k_i}$
with real polynomials $f,g$ and $h_i$ and real constants $k_i$.
\end{description}
\label{Th:CS08}
\end{thm}

\begin{rmk}
{\rm
Although \cite[Theorem~3]{CS08} is only stated for
non-degenerate centers of system~\eqref{equ:cherkas},
the proof still holds for nilpotent centers,
the same as non-degenerate centers and
nilpotent centers for Li\'enard system (\cite{Chris}).
Here we rewrite \cite[Theorem~3]{CS08} as Theorem~\ref{Th:CS08}
by revising $r(x)=x^2+O(x^3)$ as $r(x)=x^{2\nu}+O(x^{2\nu+1})$ in {\bf(C2)},
which unifies the characterization of non-degenerate centers and nilpotent centers.
}
\end{rmk}

\begin{rmk}
{\rm
As indicated in \cite[p.343]{CS08},
we can determine computationally
whether a particular system of the form \eqref{equ:cherkas}
has a center or not.
Firstly,
we check whether $P_1= 0$, i.e., case~{\bf(C1)} holds.
Secondly,
we check whether
\begin{align}
P_2P_0P_1+P_0P'_1-P_1P'_0=e P_1^3,
\label{PPPPPPP}
\end{align}
holds for some constant $e$, i.e., case~{\bf(C3)} holds.
Finally,
we seek polynomials  $A_i(x)$\,s and $r(x)$ such that
$P_i(x)=A_i(r(x))r'(x)$ for all $i=0,1,2$,
i.e., case~{\bf(C2)} holds.
}
\end{rmk}

Combining Theorem~\ref{thm:AIC}, concerning monodromy at infinity,
with Theorem~\ref{Th:CS08}, dealing with local center,
we obtain the following complete characterization of system~\eqref{equ:cherkas}
having a global center.
For convenience, we assume that in case~{\bf(C2)}
\begin{align*}
A_0(x):=\alpha_0+\cdots+\alpha_{\kappa} x^{\kappa},~
A_1(x):=\beta_0+\cdots+\beta_{\kappa} x^{\kappa},~
A_2(x):=\gamma_0+\cdots+\gamma_{\kappa} x^{\kappa}
\end{align*}
for an integer ${\kappa}\ge 1$,
where $\alpha_{\kappa}^2+\beta_{\kappa}^2+\gamma_{\kappa}^2\ne 0$.
Moreover, define vector fields
\begin{align*}
{\cal U}^{(0)}
:={\cal F}_1^{(0)}(u,v)\frac{\partial}{\partial u}
+{\cal G}_1^{(0)}(u,v)\frac{\partial}{\partial v},
~~~
{\cal V}^{(0)}
:={\cal F}_2^{(0)}(u,y)\frac{\partial}{\partial u}
+{\cal G}_2^{(0)}(u,y)\frac{\partial}{\partial y},
\end{align*}
where
{\small
\begin{align*}
{\cal F}_1^{(0)}(u,v)&:=u^{{\kappa}+2}, &{\cal G}_1^{(0)}(u,v)&:=\widetilde{A}_0(u)v^3+\widetilde{A}_1(u)v^2+\widetilde{A}_2(u)v,
\\
{\cal F}_2^{(0)}(u,y)&:=-u^{{\kappa}+2}(y+\tilde{y}_*),
&{\cal G}_2^{(0)}(u,y)&:=\widetilde{A}_0(u)+\widetilde{A}_1(u)(y+\tilde{y}_*)
+\widetilde{A}_2(u)(y+\tilde{y}_*)^2,
\end{align*}
}$\widetilde{A}_i(u):=u^{\kappa} A_i(1/u)$,
and $\tilde{y}_*:=\frac{-\beta_{\kappa}}{2\gamma_{\kappa}}$ with nonzero $\gamma_{\kappa}$.

\begin{thm}
Cherkas system~\eqref{equ:cherkas} satisfying either {\bf(W1)} or {\bf(W2)}
has a global center at the origin if and only if
$n$ is odd, $xP_0(x)<0$ for all $x\ne 0$,
and one of the following conditions holds:
\begin{description}
  \item[(G1)]
  Condition~{\bf(C1)} holds,
$a_{\ell_0}< 0$, $c_{\ell_2}< 0$, $\ell_0$ and $\ell_2$ are odd,
where $\ell_i:=\deg P_i$.

  \item[(G2)]
  Condition~{\bf(C2)} holds and either
  \begin{description}
  \item[(G2i)]
  $\gamma_{\kappa}=\beta_{\kappa}=0$, $\alpha_{\kappa}<0$,
  and the vector field ${\cal U}^{(0)}$ has no formal invariant curves of the form
  $v=\Phi(u)\in\mathbb{R}[\![u^{\frac{1}{2}}]\!]\backslash\{0\}$ such that $\Phi(0)=0$, or

\item[(G2ii)]
  $\gamma_{\kappa}<0$ and $\beta_{\kappa}^2-4\alpha_{\kappa} \gamma_{\kappa}<0$, or

\item[(G2iii)]
  $\gamma_{\kappa}<0$, $\beta_{\kappa}^2-4\alpha_{\kappa} \gamma_{\kappa}=0$,
  and the vector field ${\cal V}^{(0)}$ has no formal invariant curves of the form $v=\Psi(u)\in\mathbb{R}[\![u^{\frac{1}{2}}]\!]$ such that $\Psi(0)=0$.
\end{description}

\item[(G3)]
  Condition~{\bf(C3)} holds, $a_{\ell_0}<0$, $c_{\ell_2}<0$,
$b_{\ell_1}^2-4a_{\ell_0}c_{\ell_2}<0$ and $\ell_0+\ell_2=2\ell_1$, where $\ell_i:=\deg P_i$.
\end{description}
\label{TH:M1}
\end{thm}

{\bf Proof.}
We first consider the sufficiency.
Since $xP_0(x)<0$ for all $x\ne 0$,
the origin $(0,0)$ is the only equilibrium of system~\eqref{equ:cherkas},
which is of center-focus type as either {\bf (W1)} or {\bf(W2)} holds.
We further show that the only equilibrium $(0,0)$ is a global center
if one of conditions {\bf(G1)}-{\bf(G3)} holds.

If {\bf(G1)} holds,
then the only equilibrium $(0,0)$ is a local center
by Theorem~\ref{Th:CS08} since {\bf(C1)} holds.
To show it is a global center,
there are three cases:
\begin{center}
{\bf(L1)} $\ell_0=n$ and $\ell_2<n$,~~~
{\bf(L2)} $\ell_0=\ell_2=n$,~~~
{\bf(L3)} $\ell_0<n$ and $\ell_2=n$.
\end{center}
In case {\bf(L1)},
the Newton polygon ${\cal N}({\cal X}^{(0)})$ has exactly one edge $E^{(0)}$,
which links vertices $(0,2)$ and $(n-\ell_2,0)$.
Thus $h(E^{(0)})=2$ and $w(E^{(0)})=n-\ell_2$,
which is even since $n$ and $\ell_2$ are both odd.
Moreover,
the polynomial
\begin{align}
\mathcal{P}_{E^{(0)}}(v)=a_nv^3+c_{\ell_2}v
\label{PEG1M1}
\end{align}
has no nonzero real roots since $a_n< 0$ and $c_{\ell_2}< 0$.
It follows that condition {\bf(M1)} of Lemma~\ref{lm:AIC} holds.
In case {\bf(L2)},
since $P_1=0$ and both $a_n$ and $c_n$ are negative,
we have $b_n^2-4a_nc_n=-4a_nc_n<0$,
i.e., condition {\bf(M2)} of Lemma~\ref{lm:AIC} holds.
In case {\bf(L3)},
we have $c_n<0$ and $b_n^2-4a_nc_n=0$.
The Newton polygon ${\cal N}({\cal Y}^{(0)})$ has exactly one edge $E^{(0)}$,
linking vertices $(0,1)$ and $(n-\ell_0,-1)$.
Thus $h(E^{(0)})=2$ and $w(E^{(0)})=n-\ell_0$,
which is even since $n$ and $\ell_0$ are both odd.
Moreover,
the polynomial
\begin{align}
\mathcal{P}_{E^{(0)}}(y)=c_ny^2+a_{\ell_0}
\label{PEG1M3}
\end{align}
has no nonzero real roots since $a_{\ell_0}< 0$ and $c_n< 0$.
Thus condition {\bf(M3)} of Lemma~\ref{lm:AIC} holds.
Consequently,
system~\eqref{equ:cherkas} is monodromic at infinity and therefore
the center equilibrium $(0,0)$ is actually a global center.

If condition {\bf(G2)} holds,
then the only equilibrium $(0,0)$ is a local center by Theorem~\ref{Th:CS08}
since condition {\bf(C2)} holds.
Consider the map $\pi:\mathbb{R}^2 \to \mathbb{R}_+\times\mathbb{R}$
defined by $(x,y)\mapsto (\tilde{x},\tilde{y})=(r(x),y)$.
Then taking the pullback of the 1-form
$\omega:=\tilde{y}d\tilde{y}
+(A_0(\tilde{x})+A_1(\tilde{x})\tilde{y}+A_2(\tilde{x})\tilde{y}^2)d\tilde{x}$
associated to the system
\begin{align}
\dot {\tilde{x}}=\tilde{y},~~~
\dot y=A_0(\tilde{x})+A_1(\tilde{x})\tilde{y}+A_2(\tilde{x})\tilde{y}^2
\label{AAA}
\end{align}
along the map $\pi$,
we obtain the 1-form
$$
\pi_*\omega
=ydy+(A_0(r(x))+A_1(r(x))y+A_2(r(x))y^2)r'(x)dx
$$
associated to system~\eqref{equ:cherkas}.
The map $\pi$ discards the phase portrait of system~\eqref{AAA}
in the half-plane $\tilde{x}<0$ and creates the symmetric system~\eqref{equ:cherkas}
by ``folding'' phase portrait of system~\eqref{AAA}
in the half-plane $\tilde{x}>0$ onto the half-plane $\tilde{x}<0$
while smoothing the curves on the $\tilde{y}$-axis.
Thus the local center equilibrium $(0,0)$ of system~\eqref{equ:cherkas} is a global one
if and only if system~\eqref{AAA} has no orbits approaching to infinity
in the half-plane $\tilde{x}>0$.
Note that system~\eqref{AAA} has the same form as \eqref{equ:cherkas}.
Then using the same proof of Lemma~\ref{lm:AIC},
we can obtain that
if one of conditions {\bf(G2i)}-{\bf(G2iii)} holds,
then system~\eqref{AAA} has no orbits approaching to infinity in the half-plane $\tilde{x}>0$
and therefore system~\eqref{equ:cherkas} has a global center at the origin.

If condition {\bf(G3)} holds,
then the only equilibrium $(0,0)$ is a local center by Theorem~\ref{Th:CS08}
since condition {\bf(C3)} holds.
Note that $\ell_0+\ell_2=2\ell_1$, $\ell_i\le n$ for all $i=0,1,2$
and at least one of $\ell_i$ equals to $n$.
Then there are three cases:
\begin{center}
{\bf (L1)$'$} $\ell_0=n$ and $\ell_1, \ell_2<n$,~
{\bf (L2)$'$} $\ell_0=\ell_1=\ell_2=n$,~
{\bf (L3)$'$} $\ell_0, \ell_1<n$ and  $\ell_2=n$.
\end{center}
In case {\bf(L1)$'$},
we have $a_n<0$ and $b_n=c_n=0$, as required in {\bf(M1)}.
Moreover,
the equality $\ell_0+\ell_2=2\ell_1$ shows that $2(n-\ell_1)=n-\ell_2$.
Then, ${\cal N}({\cal X}^{(0)})$ has exactly one edge $E^{(0)}$,
linking vertices $(0,2)$ and $(2(n-\ell_1),0)$,
and the polynomial
\begin{align*}
\mathcal{P}_{E^{(0)}}(v)=a_nv^2+b_{\ell_1}v+c_{\ell_2},
\end{align*}
which has no real roots since $b_{\ell_1}^2-4a_{\ell_0}c_{\ell_2}<0$.
Thus, condition {\bf(M1)} of Lemma~\ref{lm:AIC} holds.
In case {\bf(L2)$'$}, condition {\bf(M2)} holds clearly.
In case {\bf(L3)$'$},
we have $a_n=b_n=0$ and $c_n<0$,
which ensures that $b_n^2-4a_nc_n=0$ required in {\bf(M3)}.
Moreover,
we see from the equality $\ell_0+\ell_2=2\ell_1$ that $2(n-\ell_1)=n-\ell_0$.
Then ${\cal N}({\cal Y}^{(0)})$ has exactly one edge $E^{(0)}$,
linking vertices $(0,1)$ and $(2(n-\ell_1),-1)$,
and the polynomial
\begin{align*}
\mathcal{P}_{E^{(0)}}(y)=c_ny^2+b_{\ell_1}y+a_{\ell_0}
\end{align*}
has no real roots since $b_{\ell_1}^2-4a_{\ell_0}c_{\ell_2}<0$.
Thus, condition {\bf(M3)} holds.
In each case,
system~\eqref{equ:cherkas} is monodromic at infinity
and therefore the center equilibrium $(0,0)$ of system~\eqref{equ:cherkas} is a global center.
Consequently,
we obtain the sufficiency.

Next, we consider the necessity.
Since the global center equilibrium $(0,0)$ is the only equilibrium,
we obtain that $xP_0(x)<0$ for all $x\ne 0$.
By Lemma~\ref{lm:AIC} and Theorem~\ref{Th:CS08},
$n$ is odd, one of conditions {\bf(M1)}-{\bf(M3)} holds,
and one of conditions {\bf(C1)}-{\bf(C3)} holds.

In the case that {\bf(C1)} holds,
if {\bf(M1)} is valid,
then $\ell_0=n$ and $a_{\ell_0}<0$.
Note that ${\cal N}({\cal X}^{(0)})$ has only one edge $E^{(0)}$,
linking vertices $(0,2)$ and $(n-\ell_2,0)$,
and the polynomial corresponding to the edge $E^{(0)}$ is given by \eqref{PEG1M1}.
By {\bf(M1)}, we have $n-\ell_2$ is even and
polynomial \eqref{PEG1M1} has no nonzero real roots.
It follows that $a_{\ell_0}<0$, $c_{\ell_2}<0$,
$\ell_0$ and $\ell_2$ are both odd, i.e., {\bf(G1)} holds.
If {\bf(M2)} is valid,
then $\ell_2=n$,
which is odd, $c_n<0$ and $b_n^2-4a_nc_n=-4a_nc_n<0$.
Clearly, {\bf(G1)} holds.
If {\bf(M3)} is valid,
then $\ell_2=n$, which is odd, and $c_{\ell_2}<0$.
Note that ${\cal N}({\cal Y}^{(0)})$ has only one edge $E^{(0)}$,
linking vertices $(0,1)$ with $(n-\ell_0,-1)$,
and the polynomial corresponding to edge $E^{(0)}$ is given by \eqref{PEG1M3}.
By {\bf(M3)},
we have $n-\ell_0$ is even and polynomial \eqref{PEG1M3} has no real roots.
It also implies that {\bf(G1)} holds.
Consequently,
we obtain the necessary condition {\bf(G1)} for global center.

In the case that {\bf(C2)} holds,
since one of conditions {\bf(M1)}-{\bf(M3)} holds,
system~\eqref{equ:cherkas} has no orbits approaching to infinity by Lemma~\ref{lm:AIC}.
As indicated below \eqref{AAA},
this is equivalent that
system~\eqref{AAA} has no orbits approaching to the infinity in the half-plane $\tilde{x}>0$.
Since system~\eqref{AAA} has the same form as \eqref{equ:cherkas},
using the same proof of Lemma~\ref{lm:AIC},
we can obtain that one of the conditions {\bf(G2i)}-{\bf(G2iii)} holds.
Hence, we obtain the necessary condition {\bf(G2)} for global center.

In the case that {\bf(C3)} holds,
equality~\eqref{PPPPPPP} is satisfied for a nonzero constant $e$.
We see from degrees of both sides of \eqref{PPPPPPP} that
$\ell_0+\ell_1+\ell_2=3\ell_1$, i.e.,
\begin{align}
\ell_0+\ell_2=2\ell_1.
\label{l0l1l2}
\end{align}
In the case $e=1/4$,
system~\eqref{equ:cherkas} has an unbounded invariant algebraic curve $y+2P_0(x)/P_1(x)=0$,
as shown just below equality (3.17) in \cite{CS08},
where we note that $P_0/P_1$ is a polynomial guaranteed by (3.15) in \cite{CS08}.
The existence of such an unbounded invariant algebraic curve
implies that system~\eqref{equ:cherkas} has no global center
when $e=1/4$.
In the opposite case $e\ne 1/4$,
we see from the coefficients of leading terms of both sides of \eqref{PPPPPPP} that
$c_{\ell_2}a_{\ell_0}=e b_{\ell_1}^2$ and therefore
\begin{align}
b_{\ell_1}^2-4a_{\ell_0}c_{\ell_2}\ne 0.
\label{c3DELne0}
\end{align}
If {\bf(M1)} holds,
then $\ell_0=n$ and therefore $2(n-\ell_1)=n-\ell_2$ by \eqref{l0l1l2}.
It follows that
${\cal N}({\cal X}^{(0)})$ has only one edge $E^{(0)}$,
linking vertices $(0,2)$ and $(n-\ell_2,0)$,
and
$$
\mathcal{P}_{E^{(0)}}(v)=a_{\ell_0}v^2+b_{\ell_1}v+c_{\ell_2}.
$$
Then \eqref{c3DELne0} and {\bf(M1)} ensure that
\begin{align}
a_{\ell_0}<0,~~~c_{\ell_2}<0,~~~b_{\ell_1}^2-4a_{\ell_0}c_{\ell_2}< 0.
\label{c3DEL<0}
\end{align}
If {\bf(M2)} holds,
then $c_n<0$ and $b_n^2-4a_nc_n<0$,
which implies that $a_n<0$.
Thus $\ell_0=\ell_2=n$ and therefore $\ell_1=n$ because of \eqref{l0l1l2}.
Hence \eqref{c3DEL<0} still holds.
If {\bf(M3)} holds,
then $\ell_2=n$ and therefore $2(n-\ell_1)=n-\ell_0$ by \eqref{l0l1l2}.
It follows that
${\cal N}({\cal Y}^{(0)})$ has only one edge $E^{(0)}$,
linking vertices $(0,1)$ and $(n-\ell_0,-1)$, and
$$
\mathcal{P}_{E^{(0)}}(y)=c_ny^2+b_{\ell_1}y+a_{\ell_0}.
$$
Then we see from \eqref{c3DELne0} and {\bf(M3)} that \eqref{c3DEL<0} holds.
Consequently,
we obtain from \eqref{l0l1l2} and \eqref{c3DEL<0}
the necessary condition {\bf(G3)} for global center,
and the proof of Theorem~\ref{TH:M1} is completed.
\qquad$\Box$

\section{Non-isochronicity of the global center}
\setcounter{equation}{0}
\setcounter{lm}{0}
\setcounter{thm}{0}
\setcounter{rmk}{0}
\setcounter{df}{0}
\setcounter{cor}{0}
\setcounter{pro}{0}

In the last section
we obtain global center conditions for Cherkas system~\eqref{equ:cherkas}.
This section is devoted to further show the non-isochronicity
by investigating the period function of orbits near infinity.
For this purpose,
we consider a closed rectangle as a compactification of the real plane,
as seen in Fig.~\ref{fig:TC}(a).
As analyzed in section 2,
the boundary of the rectangle is a polycycle
consisting of a finite union of equilibria and integral curves
when system~\eqref{equ:cherkas} has a global center at the origin.
After desingularization,
the boundary of the rectangle becomes a polycycle with
only hyperbolic saddles and semi-hyperbolic saddles on it.
In order to study the period of orbit near infinity,
we need to consider the passage time near each saddle corner of the polycycle
and that along each side of integral curve between two adjacent saddle corners.
Some notations and results given in \cite{LSZ} will be used in the following.

Along a side of integral curve of the polycycle,
by the Flow Box Theorem (\cite[Theorem~1.12]{DLA}) and a time-reversing if necessary,
there are analytic coordinates $(x,y)$ such that
the side lies on the $x$-axis
and the vector field along this side takes the form
\begin{align}
\frac{1}{y^qg_1(x,y)}\frac{\partial }{\partial x},
\label{vf:seg}
\end{align}
where $q$ is an integer and $g_1$ is an analytic function
satisfying that $g_1(x,0)> 0$ for all $x\in I$, a compact interval.
Near a hyperbolic saddle corner of the polycycle,
by the analytic normal form result given in \cite[Theorem~2.15]{DLA}
and a time-reversing if necessary,
there exist analytic coordinates $(x,y)$ under which
the vector field near this saddle corner takes the form
\begin{align}
\frac{1}{x^p y^q g_2(x,y)}
\left\{
x\frac{\partial }{\partial x}-y(\lambda+o(1))\frac{\partial }{\partial y}
\right\},
\label{vf:h-saddle}
\end{align}
where $p,q$ are integers, $\lambda>0$ and
$g_2$ is an analytic function such that $g_2(0,0)> 0$.
Near a semi-hyperbolic saddle corner of the polycycle,
by the $C^\infty$ normal formal result given in \cite[Theorem~2.19]{DLA}
and a time-reversing if necessary,
there exist $C^\infty$ coordinates $(x,y)$ under which
the vector field near this semi-hyperbolic saddle corner is of the form
\begin{align}
\frac{1}{x^py^qg_3(x,y)}
\left\{
(x^k+ax^{2k-1})\frac{\partial}{\partial x}-\lambda y\frac{\partial}{\partial y}
\right\},
\label{vf:sh-saddle}
\end{align}
where integer $k\ge 2$, $a\in\mathbb{R}$, $\lambda>0$ and
$g_3$ is a $C^\infty$ function with $g_3(0,0)> 0$.
Note that
factors of the form $1/(x^py^q)$ in \eqref{vf:seg}-\eqref{vf:sh-saddle} are unavoidable
because those factors appear naturally
when we consider polynomial vector fields in local coordinates near the infinity
and a time-rescaling cannot be performed to cancel those factors.
When concerning about the period,
we have to use conjugate relation rather than equivalence relation.

\begin{figure}[h]
    \centering
    \subcaptionbox{%
     }{\includegraphics[height=1.5in]{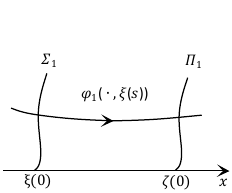}}~~~~~~~~~~~~~~~~~~
     \subcaptionbox{%
     }{\includegraphics[height=1.5in]{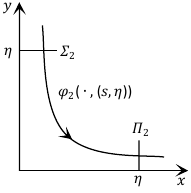}}
   \caption{Transverse sections for vector fields \eqref{vf:seg} and \eqref{vf:h-saddle}.}
    \label{fig:ts}
\end{figure}

For vector field \eqref{vf:seg},
let $\xi,\zeta:[0,1]\to \mathbb{R}^2$ be two $C^\infty$ curves
transverse to the $x$-axis such that $\xi(0)$ and $\zeta(0)$ are two points
lying in the interval $I$,
as shown in Fig.~\ref{fig:ts}(a).
Let $\Sigma_1$ and $\Pi_1$ denote the images of $\xi$ and $\zeta$, respectively.
Suppose that $\varphi_1(t,\xi(s))$ is the orbit
passing through the point $\xi(s)\in\Sigma_1$ for small $s>0$.
Let $t_1(s)$ be the passage time along the side between $\Sigma_1$ and $\Pi_1$, i.e.,
$$
t_1(s):=\min\{t\in\mathbb{R}_+:\varphi_1(t,\xi(s))\in \Pi_1\}.
$$
For vector field \eqref{vf:h-saddle},
consider transverse sections
\begin{align*}
\Sigma_2:=\{(x,y)\in\mathbb{R}^2:y=\eta\}~~~\mbox{and}~~~
\Pi_2:=\{(x,y)\in\mathbb{R}^2:x=\eta\}
\end{align*}
with small $\eta>0$, as shown in Fig.~\ref{fig:ts}(b).
Suppose that $\varphi_2(t,(s,\eta))$ is the orbit
passing through the point $(s,\eta)\in \Sigma_2$ for small $s>0$.
Let $t_2(s)$ be the passage time at the hyperbolic saddle corner, i.e.,
$$
t_2(s):=\min\{t\in\mathbb{R}_+:\varphi_2(t,(s,\eta))\in \Pi_2\}.
$$
For vector field \eqref{vf:sh-saddle},
we can define similarly the passage time $t_3(s)$ at the semi-hyperbolic saddle corner.
The following lemma (\cite[Lemma~2.4]{LSZ})
characterizes orders of the above passage times as $s$ approaches to $0^+$.

\begin{lm}
As $s\to 0^+$,
passage times $t_1(s)$, $t_2(s)$ and $t_3(s)$ satisfy that
$$
t_1(s)=O(s^q),~~~t_2(s)=O(s^{\rho}|\ln s|^\alpha),
$$
where $\rho:=\min\{p,\lambda q\}$ and
$\alpha$ equals to either $1$ if $p-\lambda q=0$, or $0$ otherwise, and
\begin{equation*}
t_3(s)=
\left\{
\begin{array}{lllll}
O(s^p),                                &q>0,
\\
O(1),                                  &q=0,~p\ge k,
\\
O(s^{p-k+1}|\ln s|^\beta),             &q=0,~p\le k-1,
\\
O(s^{-a\lambda q}\exp(\frac{\lambda q}{1-k}s^{1-k})), &q<0,
\end{array}
\right.
\end{equation*}
where $\beta$ equals to either $1$ if $p=k-1$, or $0$ otherwise.
\label{lm:time}
\end{lm}

Clearly,
the passage time $t_1(s)$ approaches to either $0$ or $+\infty$ if $q\ne 0$.
It is the same for $t_2(s)$ and $t_3(s)$ if neither $p$ nor $q$ is $0$.
When considering Cherkas system~\eqref{equ:cherkas} at infinity,
we would deal with vector fields of forms \eqref{vf:seg}-\eqref{vf:sh-saddle}
with rational $p$ and $q$.
Note that results in Lemma~\ref{lm:time} still hold
even for real $p$ and $q$,
since the proof of Lemma~\ref{lm:time} is independent of the types of $p$ and $q$.

\begin{thm}
Cherkas system~\eqref{equ:cherkas} has no isochronous global centers.
\label{th:iso}
\end{thm}

{\bf Proof.}
In order to prove this theorem,
we investigate the period of orbit of system~\eqref{equ:cherkas} near infinity.
For this purpose,
we consider a closed rectangle as a compactification of the real plane,
as seen in Fig.~\ref{fig:TC}(a).
Under transformation {\bf(T1)}, {\bf(T2)} and {\bf (T3)},
the vector field generated by system~\eqref{equ:cherkas}
near the line $\mathbb{R}\times\{\infty\}$,
the line $\{\infty\}\times\mathbb{R}$ and the point $\{\infty\}\times\{\infty\}$
are given by
$$
\widehat {\cal Z}^{(0)}:=\frac{1}{v} {\cal Z}^{(0)},~~~
\widehat {\cal Y}_*^{(0)}:=\frac{1}{u^n}{\cal Y}_*^{(0)}~~~\mbox{and}~~~
\widehat {\cal X}^{(0)}:=-\frac{1}{u^nv}{\cal X}^{(0)},
$$
respectively,
where ${\cal Z}^{(0)}$, ${\cal Y}_*^{(0)}$ and ${\cal X}^{(0)}$ are vector fields
generated by systems \eqref{PP1}, \eqref{PP2} and \eqref{PP3}, respectively.


\begin{figure}[h]
    \centering
     \subcaptionbox{%
     }{\includegraphics[height=1.5in]{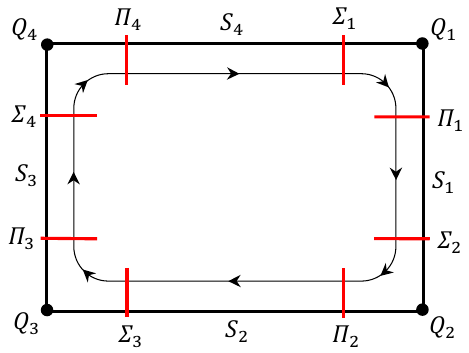}}~~~~~~
     \subcaptionbox{%
     }{\includegraphics[height=1.5in]{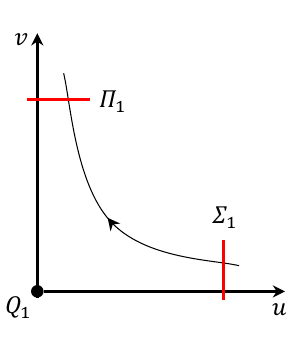}}~~~~~~
     \subcaptionbox{%
     }{\includegraphics[height=1.5in]{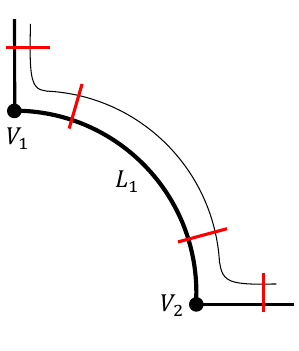}}
    \caption{
    (a) Rectangle compactification of $\mathbb{R}^2$ and transverse sections of the boundary,
    (b) saddle corner at $Q_1$ in the local coordinates,
    (c) desingularization of the degenerate saddle corner.}
    \label{fig:L1}
\end{figure}


By Theorem~\ref{TH:M1},
if system~\eqref{equ:cherkas} has a global center at the origin,
then either {\bf(G1)}, or {\bf(G2)}, or {\bf(G3)} holds.
When {\bf(G1)} holds,
we have $P_1=0$ and there are three cases
\begin{center}
{\bf(L1)} $\ell_0=n$ and $\ell_2<n$,~~~
{\bf(L2)} $\ell_0=\ell_2=n$,~~~
{\bf(L3)} $\ell_0<n$ and $\ell_2=n$.
\end{center}
In the case {\bf(L1)},
as shown in Fig.~\ref{fig:L1}(a),
the four vertices $Q_1,...,Q_4$ of the rectangle,
labeled clockwise from the top right-hand corner,
are all equilibria of system~\eqref{equ:cherkas} at infinity.
For $i=1,...,4$,
let $\xi_i,\zeta_i:[0,1]\to\mathbb{R}^2_\infty$ be analytic curves transverse to
segments $\overline{Q_{i-1}Q_i}$ and $\overline{Q_iQ_{i+1}}$ near the vertex $Q_i$
such that $\xi_i(0)\in\overline{Q_{i-1}Q_i}$ and $\zeta_i(0)\in\overline{Q_iQ_{i+1}}$,
respectively, where $Q_{0}:=Q_4$ and $Q_5:=Q_1$.
Let $\Sigma_i$ and $\Pi_i$ denote the images of $\xi_i$ and $\zeta_i$, respectively.
Then the boundary of the rectangle are divided into four corners at $Q_i$
and four linear segments $S_i$ between the corners at $Q_{i}$ and $Q_{i+1}$,
as shown in Fig.~\ref{fig:L1}(a).
Similarly to the third paragraph of this section,
we can define passage time $t_i(s)$ at the corner $Q_i$ and
passage time $\tau_i(s)$ along the side $S_i$.
Then the non-isochronicity follows directly from the limits
\begin{align}
\lim_{s\to 0^+}t_i(s)=\lim_{s\to 0^+}\tau_i(s)=0~~~\mbox{for all}~i=1,2,3,4.
\label{tstaos}
\end{align}
Actually,
we only need to consider the corner at $Q_1$ and sides $S_1$ and $S_2$
since it is similar to consider the others.
Along the side $S_1$,
the vector field generated by system~\eqref{equ:cherkas} is given by
\begin{align*}
\widehat{{\cal Y}}_*^{(0)}=\frac{1}{u^n}{\cal Y}_*^{(0)}
=\frac{1}{u^n}
\left\{
-u^{n+2}y\frac{\partial }{\partial u}
+(a_n+O(u))\frac{\partial}{\partial y}
\right\}.
\end{align*}
Along the side $S_2$,
the vector field generated by system~\eqref{equ:cherkas} is given by
\begin{align*}
\widehat{{\cal Z}}^{(0)}=\frac{1}{v}{\cal Z}^{(0)}
=\frac{1}{v}
\left\{
\frac{\partial }{\partial x}
-(P_2(x)v+P_0(x)v^3)\frac{\partial}{\partial v}
\right\}.
\end{align*}
Then Lemma~\ref{lm:time} ensures that
$$
\lim_{s\to 0^+}\tau_1(s)=\lim_{s\to 0^+}\tau_2(s)=0.
$$
Near the degenerate saddle corner at $Q_1$,
the vector field generated by system~\eqref{equ:cherkas} is given by
{\small
\begin{align*}
\widehat{\cal X}^{(0)}
=\frac{-1}{u^nv}{\cal X}^{(0)}
=\frac{-1}{u^nv}
\left\{
u^{n+2}\frac{\partial }{\partial u}+
v\big((a_n+O(u))v^2+c_{\ell_2}u^{n-\ell_2}+O(u^{n-\ell_2+1})\big)\frac{\partial}{\partial v}
\right\}.
\end{align*}
}Blowing up the degenerate equilibrium of $\widehat{\cal X}^{(0)}$
in the positive $v$-axis by the transformation
$u=w_1z_1$ and $v=z_1^{p^{(0)}}$ with $p^{(0)}=(n-\ell_2)/2$,
we obtain
\begin{align}
\widehat{{\cal X}}_h^{(1)}=\frac{a_n}{p^{(0)}w_1^nz_1^{n-p^{(0)}}}
\left\{
w_1(-1+o(1))\frac{\partial}{w_1}+z_1(1+o(1))\frac{\partial }{\partial z_1}
\right\},
\end{align}
which has a hyperbolic saddle at the origin.
On the other hand,
blowing up
in the positive $u$-axis
by the transformation $u=u_1$ and $v=u_1^{p^{(0)}}v_1$,
we obtain
\begin{align}
\widehat{{\cal X}}^{(1)}=\frac{-1}{u_1^{n-p^{(0)}}v_1}
\left\{
u_1^{n+2-2p^{(0)}}\frac{\partial }{\partial u_1}
+v_1\big(c_{\ell_2}+a_nv_1^2+O(u_1)\big)\frac{\partial}{\partial v_1}
\right\},
\end{align}
which has the only (semi-hyperbolic) equilibrium $(0,0)$ on the $v_1$-axis.
Thus the degenerate equilibrium $(0,0)$ of $\widehat{\cal X}^{(0)}$ is desingularized,
and the hyperbolic sector of $\widehat{\cal X}^{(0)}$ becomes
a poly-arc with two hyperbolic saddles at vertices, as shown in Fig.~\ref{fig:L1}(b) and (c).
Let $V_1$ and $V_2$ denote the two vertices of the poly-arc,
labeled in the clockwise direction.
Near each vertex we can choose two transverse sections as we have done at $Q_1$
and therefore
those transverse sections divide the poly-arc into $2$ corners at $V_1$ and $V_2$
and $1$ side between them.
Near corners at $V_1$ and $V_2$ the vector fields are given by
$\widehat{{\cal X}}_h^{(1)}$ and $\widehat{{\cal X}}^{(1)}$, respectively.
By Lemma~\ref{lm:time}, passage times at the two corners both approach to $0$.
The vector field along the side between corners at $V_1$ and $V_2$ is given by
$\widehat{\cal X}^{(1)}$ for $v_1\in I'$, a compact sub-interval of $(0,+\infty)$.
Thus the passage time along this side also approaches to $0$.
Consequently,
$$
\lim_{s\to 0^+}t_1(s)=0
$$
and therefore the claimed \eqref{tstaos} is proved.
Thus the global center is not isochronous in the case {\bf (L1)}.

In the case {\bf(L2)},
the four vertices $Q_1,...,Q_4$
are all equilibria of system~\eqref{equ:cherkas} at infinity.
Near the saddle corner at $Q_1$,
the vector field $\widehat{{\cal X}}^{(0)}$ is of the form
\begin{align*}
\widehat{{\cal X}}^{(0)}=\frac{-1}{u^nv}
\left\{
u^{n+2}\frac{\partial}{\partial u}
+v(c_n+o(1))\frac{\partial}{\partial v}
\right\}.
\end{align*}
By Lemma~\ref{lm:time},
the passage time $t_1(s)$ approaches to zero.
Similar to the above, $\tau_1(s)$ and $\tau_2(s)$ both approach to zero.
Thus \eqref{tstaos} is proved and therefore
the global center is not isochronous in the case {\bf (L2)}.

In the case {\bf(L3)},
the vector field
\begin{align*}
\widehat{{\cal Y}}_*^{(0)}
=\frac{1}{u^n}
\left\{
-u^{n+2}y\frac{\partial }{\partial u}
+\big(a_{\ell_0}u^{n-\ell_0}+O(u^{n-\ell_0+1})+(c_n+o(1))y^2\big)\frac{\partial}{\partial y}
\right\}.
\end{align*}
has the only equilibrium $(0,0)$ on the $y$-axis.
So on the boundary of the rectangle,
there are another two equilibria $Q_\pm:(\pm\infty,0)$ except for $Q_1,...,Q_4$,
as shown in Fig.~\ref{fig:L3}(a).
Similar to the case {\bf(L1)},
all passage times
near the hyperbolic saddle corners at $Q_1,...,Q_4$ and
along sides between any two adjacent corners
in the corner sequence $(Q_1,Q_+,Q_2,Q_3,Q_-,Q_4, Q_1)$
approach to zero.
However,
the equilibrium $(0,0)$ of $\widehat{\cal Y}_*^{(0)}$ is a degenerate equilibrium
with zero linear part.
Blowing up the degenerate equilibrium $(0,0)$
in the positive $y$-axis by the transformation
$u=w_1z_1$ and $y=z_1^{p^{(0)}}$ with $p^{(0)}=(n-\ell_0)/2$,
we obtain
\begin{align*}
\widehat{{\cal Y}}^{(1)}_{*h}=\frac{c_n}{p^{(0)}w_1^nz_1^{n-p^{(0)}}}
\left\{
w_1(-1+o(1))
\frac{\partial}{\partial w_1}
+z_1(1+o(1))\frac{\partial}{\partial z_1}
\right\},
\end{align*}
which has a hyperbolic saddle at the origin.
Moreover,
the blowing-up transformation $u=w_1z_1$ and $y=-z_1^{p^{(0)}}$ in the negative $y$-axis
brings $\widehat{\cal Y}_*^{(0)}$ into the form $-\widehat{{\cal Y}}^{(1)}_{*h}$,
which also has a hyperbolic saddle at the origin.
On the other hand,
blowing up in the positive $u$-axis by the transformation
$u=u_1$ and $y=u_1^{p^{(0)}}v_1$, we obtain
\begin{align*}
\widehat{{\cal Y}}^{(1)}_*=\frac{1}{u_1^{n-p^{(0)}}}
\left\{
-u_1^{n+2}v_1\frac{\partial}{\partial u_1}
+(a_{\ell_0}+c_nv_1^2+O(u_1))\frac{\partial}{\partial v_1}
\right\},
\end{align*}
which has no equilibria on the $v_1$-axis.
Thus the degenerate equilibrium $(0,0)$ of $\widehat{\cal Y}_*^{(0)}$ is desingularized,
and the hyperbolic sector of $\widehat{\cal Y}_*^{(0)}$ becomes
a poly-arc with hyperbolic saddles at vertices $V_1$ and $\tilde{V}_1$,
as shown in Fig.~\ref{fig:L3}(b) and (c).
The vector field near the corner at $V_1$ is given by $\widehat{{\cal Y}}^{(1)}_{*h}$,
near the corner at $\tilde{V}_1$ is given by $-\widehat{{\cal Y}}^{(1)}_{*h}$,
and along the side between corners at $V_1$ and $\tilde{V}_1$ is given by
$\widehat{{\cal Y}}_*^{(1)}$.
By Lemma~\ref{lm:time},
passage times near the corners $V_1$ and $\tilde{V}_1$ and
along the side between them all approach to zero.
Thus the passage time of \eqref{equ:cherkas} near the corner at $Q_+$ approaches to zero
and therefore the global center is not isochronous in the case {\bf (L3)}.
Consequently,
system~\eqref{equ:cherkas} has no isochronous global center when {\bf(G1)} holds.


\begin{figure}[H]
    \centering
     \subcaptionbox{%
     }{\includegraphics[height=1.5in]{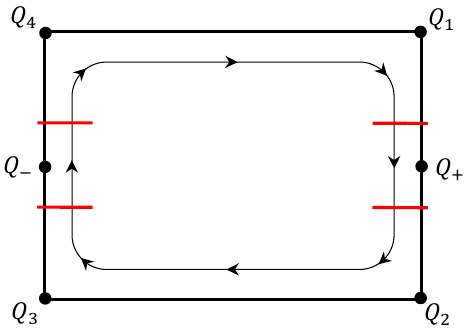}}~~~~~~~~~
     \subcaptionbox{%
     }{\includegraphics[height=1.5in]{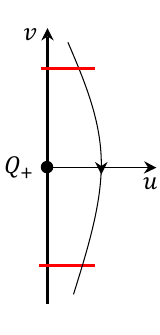}}~~~~~~~~~
     \subcaptionbox{%
     }{\includegraphics[height=1.5in]{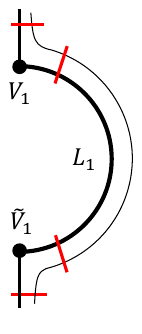}}
    \caption{
    (a) Transverse sections at the new equilibria $Q_\pm$,
    (b) Saddle corner at $Q_+$ in the local coordinates,
    (c) desingularization of the saddle corner at $Q_+$.}
    \label{fig:L3}
\end{figure}


When {\bf (G2)} holds,
we have
$$
P_i(x)=A_i(r(x))r'(x)~~~\mbox{for all}~i=0,1,2,
$$
where $A_i(x)$ and $r(x)$ are polynomials satisfying that
$A_0(0)<0$ and $r(x)=x^{2\nu}+O(x^{2\nu+1})$.
If the origin is an isochronous center,
then it is non-degenerate (\cite[Theorem~2.2]{CD97}) and therefore $\nu=1$.
Moreover,
global center implies that $xP_0(x)<0$ for all $x\ne 0$ and therefore
$xr'(x)>0$ for all $x\ne 0$.
Then $r(x)$ takes the form
$$
r(x)=x^2+\cdots+\delta x^{2\varpi}
$$
for some positive constant $\delta$ and integer $\varpi$.
Note that the polynomial map $r:\mathbb{R}_+\to\mathbb{R}_+$ is an analytic diffeomorphism.
Let $r^{-1}$ be the inverse and $R(u):=r'(r^{-1}(\frac{1}{u}))$.
Then, near $u=0$,
the function $R(u)$ can be written as
$$
R(u)=u^{-\frac{2\varpi-1}{2\varpi}}R_1(u)
$$
with $R_1(u):=2 \varpi\delta^{\frac{1}{2\varpi}}+o(1)$.
Under the transformation $(\tilde{x},\tilde{y})=(r(x),y)$,
system~\eqref{equ:cherkas} is changed into the system
\begin{align}
\dot{\tilde{x}}=r'(r^{-1}(\tilde{x}))\tilde{y},~~~
\dot{\tilde{y}}=r'(r^{-1}(\tilde{x}))
\big(A_0(\tilde{x})+A_1(\tilde{x})\tilde{y}+A_2(\tilde{x})\tilde{y}^2\big).
\label{equ:C-sys}
\end{align}
Note that system~\eqref{equ:cherkas} in the half-plane $x>0$
is conjugated to system~\eqref{equ:C-sys} in the half-plane $\tilde{x}>0$.
So $Q_1, Q_2$ and $Q_+$ (if exists) are all equilibria of system~\eqref{equ:C-sys}
at infinity of the half-plane $\tilde{x}>0$.
As analyzed above,
we only need to consider passage times of system~\eqref{equ:cherkas}
near the corner at $Q_1$ and the corner at $Q_+$ (if exists).
The above conjugation implies that
we only need to consider the same problem for system~\eqref{equ:C-sys}.
Similar to system~\eqref{equ:cherkas},
along the line $\{\infty\}\times\mathbb{R}$,
the vector field generated by system~\eqref{equ:C-sys} is given by
\begin{align*}
\widehat{{\cal V}}_*^{(0)}=\frac{R_1(u)}{u^{\kappa+\frac{2\varpi-1}{2\varpi}}} {\cal V}_*^{(0)}
\end{align*}
where as defined just before Theorem~\ref{TH:M1},
\begin{align*}
{\cal V}_*^{(0)}:=&
-u^{\kappa+2} y \frac{\partial }{\partial u}
+\big(
\tilde{A}_0(u)+\tilde{A}_1(u)y+\tilde{A}_2(u)y^2
\big)
\frac{\partial }{\partial y}
\\
=&-u^{\kappa+2} y \frac{\partial }{\partial u}
+\big(\alpha_{\kappa}+\beta_{\kappa}y+\gamma_{\kappa}y^2+O(u)\big)
\frac{\partial }{\partial y}.
\end{align*}
Near the point $\{\infty\}\times\{\infty\}$,
the vector field generated by system~\eqref{equ:C-sys} is of the form
\begin{align*}
\widehat{{\cal U}}^{(0)}=-\frac{R_1(u)}{u^{\kappa+\frac{2\varpi-1}{2\varpi}}v}{\cal U}^{(0)}
\end{align*}
where as defined just before Theorem~\ref{TH:M1},
\begin{align*}
{\cal U}^{(0)}:=&u^{\kappa+2}\frac{\partial }{\partial u}
+v\big(\tilde{A}_2(u)++\tilde{A}_1(u)v+\tilde{A}_0(u)v^2\big)\frac{\partial }{\partial v}
\\
=&u^{\kappa+2}\frac{\partial }{\partial u}
+v\big(\gamma_{\kappa}+\beta_{\kappa}v+\alpha_{\kappa}v^2+O(u)\big)\frac{\partial }{\partial v}.
\end{align*}
If {\bf(G2)} holds then either {\bf(G2i)}, or {\bf(G2ii)} or {\bf(G2iii)} holds.
When {\bf(G2i)} holds,
$\widehat{\cal V}_*^{(0)}$ has no equilibria on the $y$-axis.
So we only need to consider
the passage time of system~\eqref{equ:cherkas} near the corner at $Q_1$ or
equivalently the passage time near the degenerate saddle corner of $\widehat{\cal U}^{(0)}$
at the origin $(0,0)$.
When {\bf(G2i)} holds,
${\cal U}^{(0)}$ has no formal invariant curves of the form $v=\Phi(u)\in\mathbb{R}[\![u^{\frac{1}{2}}]\!]\backslash\{0\}$ such that $\Phi(0)=0$.
Similarly to Theorem~\ref{TH:M1} and Lemma~\ref{lm:AIC},
we have the following fact.

\noindent
{\bf Fact~1.}
{\it
There is an integer $j_*\ge 0$ such that
the Newton polygon ${\cal N}({\cal U}^{(j)})$ has exactly one edge $E^{(j)}$,
whose height is 2 and width is $2p^{(j)}$ for an integer $p^{(j)}$,
and the polynomial $\mathcal{P}_{E^{(j)}}$ has the only nonzero real root $\phi^{(j)}$,
which is of multiplicity 2, for all $j=0,...,j_*-1$,
and ${\cal N}({\cal U}^{(j_*)})$ has exactly one edge $E^{(j_*)}$,
whose height is $2$ and width is $p^{(j_*)}$ for an integer $p^{(j_*)}$,
and the polynomial $\mathcal{P}_{E^{(j_*)}}$ has no nonzero real roots,
where ${\cal U}^{(j+1)}:={\cal D}({\cal U}^{(j)};p^{(j)},1,\phi^{(j)})$
for all $j=0,...,j_*-1$
and the definition of height $($and width$)$ of an edge is given just below \eqref{defLE}.
}

It is worth mentioning that
different from the integer $p^{(i_*)}$ in {\bf(M1)},
the integer $p^{(j_*)}$ in {\bf Fact~1} is not required to be even
since we focus on the half-plane $u\ge 0$.

Similarly to the desingularization of ${\cal X}^{(0)}$
given in the proof of the sufficiency of Lemma~\ref{lm:Lien-M} when {\bf(M1)} holds,
applying successive quasi-homogeneous blowing ups according to {\bf Fact~1},
we find that after desingularizing
the hyperbolic sector of $\widehat{\cal U}^{(0)}$ in the first quadrant becomes
a poly-arc with hyperbolic saddles at vertices,
as shown in Fig.~\ref{fig:G2i}.
Let $V_1,...,V_{j_*+1},\tilde{V}_{j_*+1}...,\tilde{V}_1$ denote those vertices of
the poly-arc, labeled in the clockwise direction.
Near each vertex we can choose two transverse sections as we have done at $Q_1$
and therefore
those transverse sections divide the poly-arc into $2j_*+2$ corners and
$2j_*+1$ sides, each of which lies between two adjacent corners.
Let $L_1,...,L_{j_*},L_{j_*+1},\tilde{L}_{j_*},...,\tilde{L}_1$ denote
those sides, labeled in the clockwise direction.
We only need to consider corners at $V_1,...,V_{j_*+1}$ and
sides $L_1,...,L_{j_*}$ and $L_{j_*+1}$ since it is similar to investigate
other vertices and sides.
For $i=1,...,j_*$,
the desingularized vector field of $\widehat{{\cal U}}^{(0)}$
near the corner at $V_i$ and along the side $L_i$ are given by
\begin{align*}
\widehat{{\cal U}}_h^{(i)}
=\frac{\Theta_k(w_i,z_i)}
{(w_i)^{\varsigma_1^{(i)}}(z_i)^{\varsigma_2^{(i)}}}
{\cal U}_h^{(i)}
~~~\mbox{and}~~~
\widehat{{\cal U}}^{(i)}
=\frac{\Lambda_i(u_i,v_i)}
{(u_i)^{\varsigma_2^{(i)}}}
{\cal U}^{(i)},
\end{align*}
respectively,
where $\Theta_1(w_1,z_1):=-\frac{\alpha_{\kappa} R_1(w_1z_1)}{p^{(0)}}$,
$\Lambda_1(u_1,v_1):=-\frac{R_1(u_1)}{\phi^{(0)}+v_1}$,
for all $i=2,...,j_*$,
\begin{align*}
\Theta_i(w_i,z_i)&:=\frac{\alpha_{\kappa} \Lambda_{i-1}(w_i z_i,(z_i)^{p^{(i-1)}})}{p^{(i-1)}},
\\
\Lambda_i(u_i,v_i)&:=\Lambda_{i-1}(u_i,(u_i)^{p^{(i-1)}}(\phi^{(i-1)}+v_i)),
\end{align*}
and for all $i=1,...,j_*$,
\begin{align*}
{\cal U}_h^{(i)}&:=
w_i(-1+o(1))\frac{\partial }{\partial w_i}
+z_i(1+o(1))\frac{\partial }{\partial z_i},
\\
{\cal U}^{(i)}&:=
(u_i)^{\varsigma_0^{(i)}}\frac{\partial }{\partial u_i}
+\big({\cal P}_{E^{(i-1)}}(\phi^{(i-1)}+v_i)+O(u_i)\big)\frac{\partial }{\partial v_i}.
\\
\varsigma_0^{(i)}&:=\kappa+2-p^{(0)}-\sum_{j=0}^{i-1}p^{(j)},
\\
\varsigma_1^{(i)}&:=\kappa+\frac{2\varpi-1}{2\varpi}-\sum_{j=0}^{i-2}p^{(j)},
\\
\varsigma_2^{(i)}&:=\kappa+\frac{2\varpi-1}{2\varpi}-\sum_{j=0}^{i-1}p^{(j)}.
\end{align*}
Moreover,
the desingularized vector field of $\widehat{{\cal U}}^{(0)}$ near the corner at $V_{j_*+1}$
and along the side $L_{j_*+1}$ are given by
\begin{align*}
\widehat{{\cal U}}_h^{(j_*+1)}
&=\frac{\alpha_{\kappa}\Lambda_{j_*}\big(w_{j_*+1}z_{j_*+1}^2,(z_{j_*+1})^{p^{(j_*)}}\big)}
{p^{(j_*)}(w_{j_*+1})^{\varsigma_2^{(j_*)}}(z_{j_*+1})^{2\varsigma_2^{(j_*)}-p^{(j_*)}}}
{\cal U}_h^{(j_*+1)},
\\
\widehat{{\cal U}}^{(j_*+1)}
&=\frac{\Lambda_{j_*}\big(u_{j_*+1}^2,(u_{j_*+1})^{p^{(j_*)}}v_{j_*+1}\big)}
{2(u_{j_*+1})^{2\varsigma_2^{(j_*)}-p^{(j_*)}}}
{\cal U}^{(j_*+1)},
\end{align*}
respectively, where
\begin{align*}
{\cal U}_h^{(j_*+1)}&:=
w_{j_*+1}(-2+o(1))\frac{\partial }{\partial w_{j_*+1}}
+z_{j_*+1}(1+o(1))\frac{\partial }{\partial z_{j_*+1}},
\\
{\cal U}^{(j_*+1)}&:=
(u_{j_*+1})^{2\varsigma_0^{(j_*)}-1-p^{(j_*)}}\frac{\partial }{\partial u_{j_*+1}}
+({\cal P}_{E^{(j_*)}}(v_{j_*+1})+O(u_{j_*+1}))
\frac{\partial }{\partial v_{j_*+1}}.
\end{align*}
One can check that for all $i=1,...,j_*$,
\begin{align*}
\Theta_i(0,0)=-\frac{\alpha_{\kappa}R_1(0)}{p^{i-1}\phi^{(0)}}\ne 0,~~~
\Lambda_i(0,v_i)=-\frac{R_1(0)}{\phi^{(0)}}\ne 0.
\end{align*}
Moreover,
the order $2\varsigma_2^{(j_*)}-p^{(j_*)}$ in
$\widehat{{\cal U}}_h^{(j_*+1)}$ and $\widehat{{\cal U}}^{(j_*+1)}$
satisfies that
\begin{align}
2\varsigma_2^{(j_*)}-p^{(j_*)}
&=2\left(\varsigma_0^{(j_*)}-2+p^{(0)}+\frac{2\varpi-1}{2\varpi}\right)-p^{(j_*)}
\nonumber
\\
\nonumber
&=(2\varsigma_0^{(j_*)}-1-p^{(j_*)})+2p^{(0)}-3+\frac{2\varpi-1}{\varpi}
\\
\nonumber
&\ge 2p^{(0)}-2+\frac{2\varpi-1}{\varpi}
\\
&> 0.
\label{F1-Key}
\end{align}
where we used the fact that the power $2\varsigma_0^{(j_*)}-1-p^{(j_*)}$
in the expression of ${\cal U}^{(j_*+1)}$ is not less than $1$.
It follows that $\varsigma_2^{(j_*)}>p^{(j_*)}/2>0$ and therefore
$$
\varsigma_1^{(k)}>\varsigma_2^{(k)}\ge \varsigma_2^{(j_*)}>0~~~\mbox{for all}~k=1,...,j_*.
$$
By Lemma~\ref{lm:time},
passage times near corners at $V_1,...,V_{j_*}$ and
along side $L_1,...,L_{j_*+1}$ all approach to $0$.
Then the global center is not isochronous when {\bf(G2i)} holds.


\begin{figure}[H]
\centering
\includegraphics[height=1.8in]{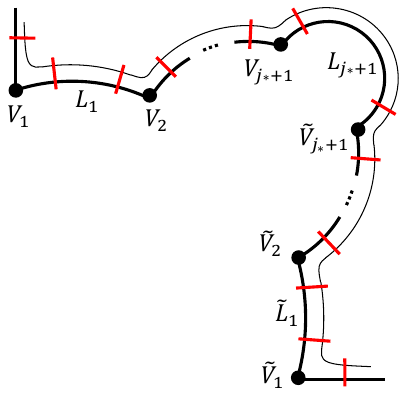}~~~
\caption{Desingularization of the degenerate saddle corner at $Q_1$.}
\label{fig:G2i}
\end{figure}


When {\bf(G2ii)} holds,
the equilibrium $(0,0)$ of $\widehat{\cal U}^{(0)}$ is a semi-hyperbolic saddle,
and $\widehat{{\cal V}}_*^{(0)}$ has no equilibria on the line $u=0$.
So Lemma~\ref{lm:time} ensures that
the passage times near the corners at $Q_1$ and $Q_2$ and
the passage time along the side between the two corners all approach to zero.
Then the global center is not isochronous when {\bf(G2ii)} holds.

When {\bf(G2iii)} holds,
the equilibrium $(0,0)$ of $\widehat{\cal U}^{(0)}$ is a semi-hyperbolic saddle.
Lemma~\ref{lm:time} shows that
the passage time near the corner at $Q_1$ approaches to zero.
On the other hand,
$\widehat{\cal V}_*^{(0)}$ has a unique equilibrium $(0,\tilde{y}_*)$ on the $y$-axis,
where $\tilde{y}_*:=\frac{-\beta_{\kappa}}{2\gamma_{\kappa}}$.
Translating it to the origin makes $\widehat{\cal V}_*^{(0)}$ be the following
$$
\widehat{\cal V}^{(0)}=\frac{R_1(u)}{u^{\kappa+\frac{2\varpi-1}{2\varpi}}} {\cal V}^{(0)},
$$
where ${\cal V}^{(0)}$ is given just before Theorem~\ref{TH:M1}.
By {\bf(G2iii)},
the vector field ${\cal V}^{(0)}$ has no formal invariant curves of the form $v=\Psi(u)\in\mathbb{R}[\![u^{\frac{1}{2}}]\!]$ such that $\Psi(0)=0$.
Similarly to Theorem~\ref{TH:M1} and Lemma~\ref{lm:AIC},
we have the following fact.

\noindent
{\bf Fact~2.}
{\it
There is an integer $j_*\ge 0$ such that
the Newton polygon ${\cal N}({\cal V}^{(j)})$ has exactly one edge $E^{(j)}$,
whose height is 2 and width is $2p^{(j)}$ for an integer $p^{(j)}$,
and the polynomial $\mathcal{P}_{E^{(j)}}$ has the only nonzero real root $\phi^{(j)}$,
which is of multiplicity 2, for all $j=0,...,j_*-1$,
and ${\cal N}({\cal V}^{(j_*)})$ has exactly one edge $E^{(j_*)}$,
whose height is $2$ and width is $p^{(j_*)}$ for an integer $p^{(j_*)}$,
and the polynomial $\mathcal{P}_{E^{(j_*)}}$ has no nonzero real roots,
where ${\cal V}^{(j+1)}:={\cal D}({\cal V}^{(j)};p^{(j)},1,\phi^{(j)})$
for all $j=0,...,j_*-1$ and
the definition of height $($and width$)$ of an edge is given just below \eqref{defLE}.
}

We mention that, similar to {\bf Fact~1},
the integer $p^{(j_*)}$ given in {\bf Fact~2} is not required to be even
since we also focus on the half-plane $u\ge 0$.
The main difference between the two facts is that
the $(0,2)$ if the left-end point of ${\cal N}({\cal U}^{(0)})$,
but $(0,1)$ for ${\cal N}({\cal V}^{(0)})$.
This would lead to the invalidity of an inequality similar to \eqref{F1-Key}
and the passage time could approach to a nonzero constant.
Therefore,
some other method will be used to prove the non-isochronicity.

Applying successive quasi-homogeneous blow-ups according to the above fact,
we find that
after desingularizing the degenerate equilibrium $(0,0)$ of
the vector field $\widehat{\cal V}^{(0)}$,
the hyperbolic sector in the half-plane $u\ge 0$ becomes
a poly-arc with hyperbolic saddles at vertices,
as shown in Fig.~\ref{fig:G2iii}.
Let $V_1,...,V_{j_*+1},\tilde{V}_{j_*+1}...,\tilde{V}_1$ denote those vertices of
the poly-arc, labeled in the clockwise direction.
Near each vertex we can choose two transverse sections as we have done at $Q_1$
and therefore
those transverse sections divide the poly-arc into $2j_*+2$ corners and
$2j_*+1$ sides, each of which lies between two adjacent corners.
Let $L_1,...,L_{j_*},L_{j_*+1},\tilde{L}_{j_*},...,\tilde{L}_1$ denote
those sides, labeled in the clockwise direction.
We only consider those vertices $V_1,...,V_{j_*+1}$ and those sides
$L_1,...,L_{j_*}$ and $L_{j_*+1}$ since it is similar to investigate
other vertices and sides.
For $i=1,...,j_*$,
the desingularized vector field of $\widehat{{\cal V}}^{(0)}$ near the corner at $V_i$
and along the side $L_i$ are given by
\begin{align*}
\widehat{{\cal V}}_h^{(i)}
=\frac{\gamma_{\kappa} R_1(w_i z_i)}
{p^{(i-1)}(w_i)^{\varsigma_1^{(i)}}(z_i)^{\varsigma_2^{(i)}}}
{\cal V}_h^{(i)}
~~~\mbox{and}~~~
\widehat{{\cal V}}^{(i)}
=\frac{R_1(u_i)}
{(u_i)^{\varsigma_2^{(i)}}}
{\cal V}^{(i)},
\end{align*}
respectively,
where
\begin{align*}
{\cal V}_h^{(i)}:=&
w_i(-1+o(1))\frac{\partial }{\partial w_i}
+z_i(1+o(1))\frac{\partial }{\partial z_i},
\\
{\cal V}^{(i)}:=&
-(u_i)^{\vartheta_0^{(i)}}\big(\tilde{y}_*+(u_i)^{p^{(0)}}(\phi^{(0)}+o(1))\big)
\frac{\partial }{\partial u_i}
\\
&+\big({\cal P}_{E^{(i-1)}}(\phi^{(i-1)}+v_i)+O(u_i)\big)
\frac{\partial }{\partial v_i},
\\
\vartheta_0^{(i)}:=&\varsigma_0^{(i)}+p^{(0)}.
\end{align*}
Moreover,
the desingularized vector field of $\widehat{{\cal V}}^{(0)}$ near the corner at $V_{j_*+1}$
and along the side $L_{j_*+1}$ are given by
\begin{align*}
\widehat{{\cal V}}_h^{(j_*+1)}
&=\frac{\gamma_{\kappa} R_1(w_{j_*+1}z_{j_*+1}^2)}
{p^{(j_*)}(w_{j_*+1})^{\varsigma_2^{(j_*)}}(z_{j_*+1})^{2\varsigma_2^{(j_*)}-p^{(j_*)}}}
{\cal V}_h^{(j_*+1)},
\\
\widehat{{\cal V}}^{(j_*+1)}
&=\frac{R_1(u_{j_*+1}^2)}
{2(u_{j_*+1})^{2\varsigma_2^{(j_*)}-p^{(j_*)}}}
{\cal V}^{(j_*+1)},
\end{align*}
respectively, where
\begin{align*}
{\cal V}_h^{(j_*+1)}
:=&w_{j_*+1}(-2+o(1))\frac{\partial }{\partial w_{j_*+1}}
+z_{j_*+1}(1+o(1))\frac{\partial }{\partial z_{j_*+1}},
\\
{\cal V}^{(j_*+1)}
:=&-(u_{j_*+1})^{2\vartheta_0^{(j_*)}-1-p^{(j_*)}}
\big(\tilde{y}_*+(u_{j_*+1})^{2p^{(0)}}(\phi^{(0)}+o(1))\big)
\frac{\partial }{\partial u_{j_*+1}}
\\
&+({\cal P}_{E^{(j_*)}}(v_{j_*+1})+O(u_{j_*+1}))
\frac{\partial }{\partial v_{j_*+1}}.
\end{align*}
In the case $\varpi\ge 2$,
since neither $\frac{2\varpi-1}{2\varpi}$ nor $\frac{2\varpi-1}{\varpi}$ is an integer,
$$
\varsigma_1^{(i)}\ne 0~~~\mbox{and}~~~\varsigma_2^{(i)}\ne 0~~~\mbox{for all}~i=1,...,j_*,
$$
and the order $2\varsigma_2^{(j_*)}-p^{(j_*)}$ of $\widehat{{\cal V}}^{(j_*+1)}$
cannot be zero.
By Lemma~\ref{lm:time},
passage times near corners at $V_1,...,V_{j_*}$ and along sides $L_1,...,L_{j_*+1}$
either approach to $0$ or $\infty$.
Thus the global center is not isochronous when $\varpi \ge 2$.


\begin{figure}[H]
\centering
\includegraphics[height=1.8in]{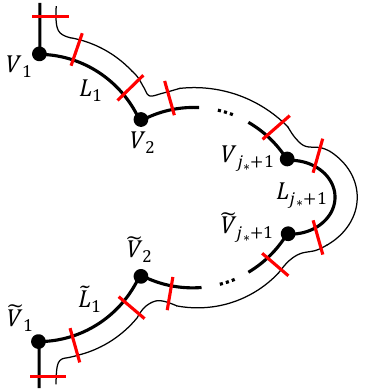}\\
\caption{Desingularization of the saddle corner at $Q_+$.}
\label{fig:G2iii}
\end{figure}


In the case $\varpi=1$,
the order $2\varsigma_2^{(j_*)}-p^{(j_*)}$ of
$\widehat{{\cal V}}^{(j_*+1)}$ can be zero,
implying that the passage time along the side $L_{j_*+1}$ approach to a nonzero constant
and the non-isochronicity cannot be induced.
So we have to use other method to prove the non-isochronicity.
In this case, we have
\begin{align*}
P_i(x)=xA_i(x^2)~~~\mbox{for all}~i=1,2,3,
\end{align*}
which are all odd functions.
Let $\wp(x):=\int_0^xP_2(\xi)d\xi$.
Then the transformation
\begin{align*}
u=\chi(x):=\int_0^x\exp(-\wp(\xi))d\xi
~~~\mbox{and}~~~
v=y\exp(-\wp(x))
\end{align*}
is globally invertible.
Under the above transformation,
system~\eqref{equ:cherkas} becomes an analytic Li\'enard system
\begin{align}
\dot u= v,~~~\dot v=g(u)+f(u)v
\label{equ:CtoL}
\end{align}
with a center equilibrium at the origin, where
$$
f(u):=P_1(\chi^{-1}(u))
~~~\mbox{and}~~~
g(u):=P_0(\chi^{-1}(u)) \exp(-\wp({\chi^{-1}(u)})).
$$
Since $P_2$ is odd,
we have $\wp$ is even and therefore $\chi$ is odd,
implying that $\chi^{-1}$ is also odd.
Thus functions $f$ and $g$ are both odd.
As proved in \cite[Corollary~1]{CD04},
Li\'enard system \eqref{equ:CtoL} has an isochronous center at the origin if and only if
\begin{align}
g(u)=u+\frac{1}{u^3}\left(\int_0^u \xi f(\xi) d\xi \right)^2,
\label{equ:L-iso}
\end{align}
which is equivalent to the equality
\begin{align*}
P_0(x)\exp(-\wp(x))=\chi(x)+\frac{1}{\chi^3(x)}\left(\int_0^{\chi(x)}\xi P_1(\chi^{-1}(\xi))d\xi\right)^2.
\end{align*}
Differentiating on both sides with respect to $x$,
we obtain
\begin{align*}
\int_0^{\chi(x)}\xi P_1(\chi^{-1}(\xi))d\xi
=\frac{P_1 \pm \sqrt{{\cal Q}}}{3}\chi^2(x),
\end{align*}
where ${\cal Q}:=3-3P_0'+3P_0P_2+P_1^2$.
Further, differentiating on both sides with respect to $x$,
we obtain
\begin{align}
\chi(x)=\frac{3P_1-2(P_1\pm\sqrt{{\cal Q}})}{\big(P_1\pm\sqrt{{\cal Q}}\big)'}\exp(-\wp(x)).
\label{equ:none-e}
\end{align}
Note that the right hand side of the above equality is clearly an elementary function.
However, the left hand side is non-elementary.
Actually, as indicated in \cite[p.971]{Rosen},
if $\chi(x)=\int_0^x\exp(-\wp(\xi))d\xi$ is elementary,
then there exist coprime polynomials $M$ and $N$ such that
$$
\int_0^x\exp(-\wp(\xi))d\xi=\frac{N(x)}{M(x)}\exp(-\wp(x)).
$$
Differentiating on both sides with respect to $x$ leads to the equality
$$
M(M-N'+NP_2)=-NM',
$$
which has no polynomial solutions $M$ and $N$ since $\deg P_2\ge 1$.
Thus $\chi$ is non-elementary and therefore \eqref{equ:none-e} cannot be hold,
implying that \eqref{equ:L-iso} cannot be hold, i.e.,
system~\eqref{equ:CtoL} has no isochronous center.
So the global center of system~\eqref{equ:cherkas} is not isochronous
in the case $\varpi=1$.
Consequently,
system~\eqref{equ:cherkas} has no isochronous global center when {\bf(G2)} holds.

When {\bf (G3)} holds,
the global center is not isochronous,
which can be proved similarly to the situation when {\bf (G1)} holds.
Thus the proof of this theorem is completed.
\qquad$\Box$

\section{Applications}
\setcounter{equation}{0}
\setcounter{lm}{0}
\setcounter{thm}{0}
\setcounter{rmk}{0}
\setcounter{df}{0}
\setcounter{cor}{0}
\setcounter{pro}{0}

Authors of \cite{GLV15} considered local center problem for the homogeneous Kukles system
\begin{align}
\dot x=y,~~~\dot y=\delta x+\sum_{i=0}^n a_{n-i,i} x^{n-i} y^i,
\label{homo-K}
\end{align}
where $\delta\le 0$, $n\ge 2$ and not all $a_{n-i,i}$\,s equal to zero.
Moreover,
global center conditions are obtained in \cite{CL24} for $\delta=-1$ and $n=5$.
Applying Theorem~\ref{TH:M1} to system~\eqref{homo-K},
we can obtain explicit conditions of global center for all $\delta$ and $n$.

\begin{thm}
System~\eqref{homo-K} has a global center if and only if it is of the form
\begin{align}
\dot x=y,~~~\dot y=\delta x+a_{n,0}x^n+a_{n-2,2}x^{n-2}y^2,
\label{GCHK}
\end{align}
where $\delta\le 0$, $n$ is odd, $a_{n,0}\le 0$, $a_{n-2,2}\le 0$ and $\delta^2+a_{n,0}^2\ne 0$.
\label{cor:hkgc}
\end{thm}

{\bf Proof.}
By Lemmas~\ref{lm:NtoC} and \ref{lm:AIC},
if system~\eqref{homo-K} has a global center,
then it is of the form
\begin{align}
\dot x=y,~~~\dot y=\delta x+a_{n,0}x^n+a_{n-1,1}x^{n-1}y+a_{n-2,2}x^{n-2}y^2
\label{homo-K-R}
\end{align}
with odd $n$.
Moreover,
we see from conditions {\bf (W1)} and {\bf (W2)} that
the equilibrium $(0,0)$ of system~\eqref{homo-K-R} is isolated and monodromic if and only if
\begin{align}
\delta\le 0,~~~a_{n,0}\le 0~~~\mbox{and}~~~\delta^2+a_{n,0}^2\ne 0.
\label{dad2a2}
\end{align}
Let ${\cal X}$ and ${\cal Y}$ denote vector fields generated
by systems~\eqref{GCHK} and \eqref{homo-K-R} respectively.
Clearly,
the monodromic equilibrium $(0,0)$ of system~\eqref{GCHK} is a center
since the system is time-reversible with respect to the $x$-axis.
If $a_{n-1,1}\ne 0$
then the wedge product ${\cal X} \wedge {\cal Y}$ satisfies that
\begin{align*}
{\cal X} \wedge {\cal Y}
&=
\left|
\begin{array}{lll}
y & \delta x+a_{n,0}x^n+a_{n-2,2}x^{n-2}y^2
\\
y & \delta x+a_{n,0}x^n+a_{n-1,1}x^{n-1}y+a_{n-2,2}x^{n-2}y^2
\end{array}
\right|
\\
&=a_{n-1,1}x^{n-1}y^2
\\
&\le 0~\mbox{(or}\ge0\mbox{)},
\end{align*}
which implies that the equilibrium $(0,0)$ of system~\eqref{homo-K-R} is a focus
by Theorem~7.2 of \cite{DLA} (or \cite[Lemma~3.1, p.207]{ZZF}).
So we only need to further seek global center conditions for
system~\eqref{homo-K-R} when $a_{n-1,1}=0$ and \eqref{dad2a2} holds.
In the case $a_{n-2,2}\ne 0$,
by Theorem~\ref{TH:M1}~{\bf(G1)},
under condition~\eqref{dad2a2},
system~\eqref{homo-K-R} with $a_{n-1,1}=0$ and
has a global center if and only if $a_{n-2,2}<0$.
In the oppositive case $a_{n-2,2}= 0$,
system~\eqref{homo-K-R} becomes a potential system,
which clearly has a global center at the origin.
Thus, we obtain the complete description on global center
for homogeneous Kukles system \eqref{homo-K} given in Theorem~\ref{cor:hkgc} and
the proof is completed.
\qquad$\Box$

It has been studied in \cite{CLZ}
the monodromy condition at infinity for polynomial Li\'enard system
\begin{align}
\dot x=y,~~~
\dot y=P_0(x)+P_1(x)y,
\label{equ:Lien}
\end{align}
where
$P_0(x):=a_0+\cdots+a_{\ell_0}x^{\ell_0}$,
$P_1(x):=b_0+\cdots+b_{\ell_1}x^{\ell_1}$
and $a_{\ell_0}b_{\ell_1}\ne 0$.
Here we provide a simple proof to show the convenience of toroidal compactification.

\begin{lm}
Li\'enard system~\eqref{equ:Lien} is monodromic at infinity if and only if
either {\bf (L1)} $\ell_0$ is odd, $\ell_0>2\ell_1+1$ and $a_{\ell_0}<0$,
or {\bf (L2)} $\ell_0=2\ell_1+1$ and $b_{\ell_1}^2+2(\ell_0+1)a_{\ell_0}<0$.
\label{lm:Lien-M}
\end{lm}

\noindent{\bf Proof.}
As done for Cherkas system~\eqref{equ:cherkas},
under the toroidal compactification,
we only need to consider systems~\eqref{PP1}-\eqref{PP3} with $P_2=0$
and they become
\begin{align}
\dot x&=1,
&\dot v&=-P_0(x) v^3- P_1(x) v^2,
\label{LPP1}
\\
\dot u&=-u^{n+2}y,
&\dot y&=\widetilde{P}_0(u)+\widetilde{P}_1(u) y,
\label{LPP2}
\\
\dot u&=u^{n+2},
&\dot v&=\widetilde{P}_0(u) v^3+ \widetilde{P}_1(u) v^2,
\label{LPP3}
\end{align}
respectively,
where $n:=\max\{\ell_0,\ell_1\}$ and $\widetilde{P}_i(u):=u^nP_i(1/u)$.
Note that system~\eqref{equ:Lien} is monodromic at infinity
if and only if
system~\eqref{LPP1} has no orbits in the region
$\{(x,v)\in\mathbb{R}^2:v\ne 0\}$
approaching to any point on the $x$-axis and
system~\eqref{LPP2} has no orbits in the region
$\{(u,y)\in\mathbb{R}^2:u\ne 0\}$
approaching to any point on the $y$-axis and moreover
the equilibrium $(0,0)$ of system~\eqref{LPP3} is a saddle.

We first consider the sufficiency.
Clearly,
system~\eqref{LPP1} has no orbits
in the region $\{(x,v)\in\mathbb{R}^2:v\ne 0\}$
approaching to any point on the $x$-axis.
If either {\bf (L1)} or {\bf(L2)} holds,
we have $\ell_0\ge 2\ell_1+1$ and therefore $n=\ell_0>\ell_1$.
The second equation in \eqref{LPP2} takes the form $\dot y=a_{\ell_0}+O(u)$.
So system~\eqref{LPP2} has no orbits in the region
$\{(u,y)\in\mathbb{R}^2:u\ne 0\}$ approaching to any point on the $y$-axis.
Since the equilibrium $(0,0)$ of system~\eqref{LPP3} is degenerate,
we blow it up in the positive $v$-direction by the transformation
$u=w_1 z_1$ and $v=z_1^{(n+1)/2}$ and obtain
\begin{align}
\dot w_1=w_1(-a_{\ell_0}+o(1)),~~~
\dot z_1=z_1(a_{\ell_0}+o(1)),
\label{Lvf:Xh1}
\end{align}
where a common factor $2z_1^{n+1}/(n+1)$ is eliminated.
The equilibrium $(0,0)$ of system~\eqref{Lvf:Xh1} is a hyperbolic saddle
whose stable manifold and unstable manifold lie on the axes.
It is similar to blow up in the negative $v$-direction
by the transformation $u=w_1 z_1$ and $v=-z_1^{(n+1)/2}$
and find that
the equilibrium $(0,0)$ of the transformed system has the same property.
Moreover,
we blow up in the $u$-direction by the transformation
$u=u_1$ and $v=u_1^{(n+1)/2}v_1$ and obtain
\begin{align}
\dot u_1=u_1,~~~
\dot v_1=v_1\big({\cal Q}(v_1)+O(u_1)\big),
\label{Lvf:X1}
\end{align}
where a common factor $u_1^{n+1}$ is eliminated and
\begin{equation*}
{\cal Q}(v_1):=\left\{
\begin{array}{llll}
a_{\ell_0} v_1^2-(\ell_0+1)/2  & \mbox{if}~\ell_0> 2\ell_1+1,
\\
a_{\ell_0} v_1^2+b_{\ell_1} v_1-(\ell_0+1)/2 & \mbox{if}~\ell_0= 2\ell_1+1.
\end{array}
\right.
\end{equation*}
If {\bf (L1)} or {\bf(L2)} holds,
then system~\eqref{Lvf:X1} has the only equilibrium $(0,0)$ on the $v_1$-axis,
which is a hyperbolic saddle.
After blowing down,
we find that the degenerate equilibrium $(0,0)$ of system~\eqref{LPP3} is a saddle.
Consequently,
system~\eqref{equ:Lien} is monodromic at infinity
if either {\bf(L1)} or {\bf(L2)} holds.

Next, we consider the necessity.
If neither {\bf (L1)} nor {\bf(L2)} holds,
then there are five cases:
{\bf(i)}   $\ell_0\le\ell_1$,
{\bf(ii)}  $\ell_1<\ell_0<2\ell_1+1$,
{\bf(iii)} $\ell_0=2\ell_1+1$ and $b_{\ell_1}^2+2(\ell_0+1)a_{\ell_0}\ge 0$,
{\bf(iv)}  $\ell_0>2\ell_1+1$ and $\ell_0$ is even, and
{\bf(v)}   $\ell_0>2\ell_1+1$, $\ell_0$ is odd and $a_{\ell_0}> 0$.

In case {\bf(i)},
we have $n=\ell_1\ge \ell_0$ and therefore
on the invariant $y$-axis
the second equation of \eqref{LPP2} is either of the form
$\dot y|_{u=0}=a_{\ell_0}+b_{\ell_1}y$ or $\dot y|_{u=0}=b_{\ell_1}y$.
Then system~\eqref{LPP2} has an orbit not on the $y$-axis connecting with the semi-hyperbolic equilibrium $(0,-a_{\ell_0}/b_{\ell_1})$ or $(0,0)$ by \cite[Theorem~2.19]{DLA} or \cite[Theorem~7.1, p.114]{ZZF}.
So system~\eqref{equ:Lien} is not monodromic at infinity,
a contradiction.

In case {\bf(ii)},
under the transformation $u=u_1$ and $v=u_1^{n-\ell_1}v_1$
together with the time-rescaling ${\rm d}t\to u_1^{2(n-\ell_1)}{\rm d}t$,
system~\eqref{LPP3} becomes
\begin{align*}
\dot u_1=u_1^{2+2\ell_1-n},~~~
\dot v_1=v_1(a_{\ell_0}v_1^2+b_{\ell_1}v_1+O(u_1))
\end{align*}
which has a semi-hyperbolic equilibrium $(0,-b_{\ell_1}/a_{\ell_0})$
since $2+2\ell_1-n=1+(2\ell_1+1)-\ell_0\ge 2$.
By \cite[Theorem~2.19]{DLA} or \cite[Theorem~7.1, p.114]{ZZF},
the above system has an orbit connecting with this equilibrium
in the half-plane $u_1>0$ and therefore
system~\eqref{LPP3} has an orbit connecting with the origin
along the line $v=-b_{\ell_1}u/a_{\ell_0}$
in the half-plane $u>0$, i.e.,
the equilibrium $(0,0)$ of system~\eqref{LPP3} is not a saddle.
So system~\eqref{equ:Lien} is not monodromic at infinity,
a contradiction.

In case {\bf(iii)},
the same transformation in case {\bf (ii)} changes system~\eqref{LPP3} as
\begin{align*}
\dot u_1=u_1,~~~
\dot v_1=v_1\{a_{\ell_0} v_1^2+b_{\ell_1} v_1-(\ell_0+1)/2+O(u_1)\}.
\end{align*}
If $b_{\ell_1}^2+2(\ell_0+1)a_{\ell_0}= 0$,
then there is only one equilibrium on the $v_1$-axis,
which is a saddle-node.
If $b_{\ell_1}^2+2(\ell_0+1)a_{\ell_0}> 0$,
there are two equilibria on the $v_1$-axis,
which are a saddle and a node.
In both two situations,
after blowing down,
we find that the equilibrium $(0,0)$ of system~\eqref{LPP3} is not a saddle,
the same contradiction as above case.

In case {\bf(iv)},
$n$ equals $\ell_0$ and is even.
On the invariant $u$-axis of system~\eqref{LPP3},
we have $\dot u|_{v=0}=u^{n+2}$,
which implies that the equilibrium $(0,0)$ is not a saddle,
the same contradiction as above.

In case {\bf(v)},
$n=\ell_0>\ell_1$, $\ell_0$ is odd and $a_{\ell_0}>0$.
On invariant axes of system~\eqref{LPP3},
we have $\dot u|_{v=0}=u^{n+2}$ and $\dot v|_{u=0}=a_{\ell_0} v^3$,
which implies that the equilibrium $(0,0)$ is not a saddle,
the same contradiction as above.
Consequently,
we obtain the necessity by contradiction.
Thus the proof of Lemma~\ref{lm:Lien-M} is completed.
\qquad$\Box$

\begin{rmk}
{\rm
One can refer to \cite{Chris} for the local center condition of
polynomial Li\'enard system~\eqref{equ:Lien} and \cite{CLZ} for the global center condition.
Further,
we claim that system~\eqref{equ:Lien} has no isochronous global centers.
Actually,
it is shown in formula (24) of \cite{CD04} that
if the origin is an isochronous center, then
$$
\ell_0=2\ell_1+1~~~\mbox{and}~~~b_{\ell_1}^2+a_{\ell_0}(\ell_0+3)^2/4=0,
$$
which contradicts to the monodromy condition given in Lemma~\ref{lm:Lien-M}.
Moreover, it is proved in \cite{CJ89} that
any nonlinear polynomial potential system has no isochronous centers.
Consequently, polynomial Newton system~\eqref{equ:Newton} has no isochronous global centers.
}
\end{rmk}


\end{document}